\documentclass[12pt]{amsart}

\usepackage{amsmath,amssymb,txfonts}
\usepackage{amssymb}
\usepackage{amsmath}
\usepackage{mathrsfs}
\usepackage{amsmath,amssymb}
\usepackage{color}
\usepackage{extarrows}

\usepackage{amsmath,amssymb,txfonts}
\usepackage{amssymb}
\usepackage{amsmath}
\usepackage{mathrsfs}
\usepackage{amsmath,amssymb}
\usepackage{color}

\newtheorem{theorem}{Theorem}[section]
\newtheorem{lemma}[theorem]{Lemma}
\newtheorem{proposition}[theorem]{Proposition}
\newtheorem{corollary}[theorem]{Corollary}
\theoremstyle{definition}
\newtheorem{definition}[theorem]{Definition}

\theoremstyle{remark}
\newtheorem{remark}[theorem]{Remark}

\renewcommand{\vec}[1]{\boldsymbol{#1}}
\numberwithin{equation}{section}

\renewcommand{\epsilon}{\eps}

\newcommand{\N}{{\mathbb N}}
\newcommand{\R}{{\mathbb R}}

\newcommand{\dvg}{{\rm div}}

\newcommand{\eps}{\varepsilon}

\newcommand{\pnorm}[2][]{\if #1'' \left|#2\right|_p \else \left|#2\right|_{#1} \fi}

\newcommand{\loc}{{\rm loc}}

\renewcommand{\theta}{\vartheta}

\numberwithin{equation}{section}

\usepackage{color}

\begin{document}

\title[Capacity \& Perimeter from $\alpha$-Hermite Bounded Variation]
{Capacity \& Perimeter from $\alpha$-Hermite Bounded Variation}


\author[Jizheng Huang]{Jizheng Huang}
\address{ School of Science, Beijing University of Posts and Telecommunications, Beijing 100876, P.R.China. }

\email{hjzheng@163.com}

%\thanks{J.Z. Huang was supported  by the National Natural
%Science Foundation of China under grants (No.11471018) and the
%Beijing Natural Science Foundation under Grant (No.1142005).}

\author[Pengtao Li]{Pengtao\ Li}
\address{College of Mathematics, Qingdao University, Qingdao, Shandong 266071, China}

\email{ptli@qdu.edu.cn}

\author{Yu Liu}
\address{School of Mathematics and Physics, University of Science and Technology Beijing, Beijing 100083, China}
\email{liuyu75@pku.org.cn}
%\thanks{Y. Liu was supported by the National Natural Science
%Foundation of China (No.10901018, 11471018), the Fundamental
%Research Funds for the Central Universities (No.FRF-TP-14-005C1) and
%Program for New Century Excellent Talents in University. }

%    author two information


\date{}

\dedicatory{}


\begin{abstract}
 Let $\mathcal{H}_{\alpha}=\Delta-(\alpha-1)|x|^{\alpha}$ be an $[1,\infty)\ni\alpha$-Hermite operator for the hydrogen atom located at the origin in $\mathbb R^d$.
 In this paper, we are motivated by the classical case $\alpha=1$ to investigate the space of functions with $\alpha$-{\it Hermite Bounded Variation} and its functional capacity and geometrical perimeter.
\end{abstract}

\maketitle

%\addcontentsline{toc}{chapter}{}
%\section{List of symbols}
%\addcontentsline{toc}{chapter}{List of symbols}
%\section{Introduction} \addcontentsline{toc}{chapter}{Introduction}

 \vspace{0.1in}

\tableofcontents \pagenumbering{arabic}

\section*{Introduction}\label{s0}

A function of bounded variation, simply a
BV-function, is a real-valued function whose total variation is
finite.  In the multi-variable setting, a function defined on an
open subset $\Omega\subseteq \mathbb R^{d}, d\geq2$, is said to have bounded
variation provided that its distributional derivative is a
vector-valued finite Radon measure over the subset $\Omega$. Precisely,
\begin{definition}
A function $u\in L^{1}(\Omega)$ whose partial derivatives in the
sense of distributions are measures with finite total variation
$\|Du\|$ in $\Omega$ is called a function of bounded variation,
where
\begin{eqnarray*}
\|Du\|:=\sup\Big\{\int_{\Omega}u\text{ div }\nu dx:\
\nu=(\nu_{1},\ldots, \nu_{d})\in C^{\infty}_{0}(\Omega; \mathbb
R^{d}), |\nu(x)|\leq 1, x\in\Omega\Big\}<\infty.
\end{eqnarray*}
The class of all such functions will be denoted by $BV(\Omega)$. The norm of $BV(\Omega)$ is defined as
$$\|u\|_{BV}:=\|u\|_{L^{1}(\Omega)}+\|Du\|.$$
\end{definition}
Note that the
BV-functions form an algebra of discontinuous functions whose first
derivative exists almost everywhere. So it is frequently and naturally utilized to define generalized solutions of
nonlinear problems involving functional analysis, ordinary and
partial differential equations, mathematical physics and
engineering. For instance, when working with minimization problems,
reflexivity or the weak compactness property of the function space
$W^{1,p}(\mathbb R^{d})$, $p>1$, usually plays an important role.
For the case of the space $W^{1,1}(\mathbb R^{d})$, one possible way
to deal with this lack of reflexivity is to consider the space
$BV(\mathbb R^{d})$.  As a wider class of functions,  the space
$BV(\mathbb R^{d})$ provides tools, such as lower semicontinuity of
the total variation measure, which can be used to overcome the
problems caused by the lack of reflexivity in the arguments. We
refer the reader to \cite{ambr}, \cite{BM}, \cite{Bramanti},
\cite{Carbonaro} and \cite{LS}.

In the study of  the pointwise behavior of a Sobolev function, the
notion of capacity plays a crucial role. The functional capacities
are of fundamental importance in various branches of mathematics
such as analysis, geometry, mathematical physics, partial
differential equations, and probability theory, see \cite{Costea},
\cite{Landis}, \cite{mazya5} and \cite{JXYZ} for the details.
 In recent years, the capacity related to bounded variation functions attracts the attentions of many researchers and  a
 lot of progress have been obtained. We refer to \cite{Ziemer} for the information of the classical BV-capacity in $\mathbb R^{d}$.
  In 2010,  Hakkarainen and Kinnunen \cite{HK} studied basic properties of the BV-capacity and the Sobolev capacity  in
   a complete metric space equipped with a doubling measure and supporting a weak Poincar\'{e} inequality. The relation  between
   the variational Sobolev 1-capacity and versions of variational BV-capacity in a complete metric space was further investigated by
     Hakkarainen  and Shanmugalingam \cite{HS}. In \cite{xiao2}, J. Xiao introduced the BV-type capacity on Gaussian spaces $\mathbb G^{d}$,
      and as an application, the Gaussian BV-capacity was used to the trace theory of Gaussian BV-space. On the generalized Grushin plane,   Liu \cite{liu} obtained some sharp trace and  isocapacity
inequalities by the BV capacity.
       For further information on this topic, we refer to \cite{HS}, \cite{La}, \cite{Xiao},  \cite{T. Wang}, \cite{xiao3} and  the references therein.

In this paper, for $\alpha\in [1,\infty)$,   let
$$\mathcal{H}_{\alpha}=\Delta-(\alpha-1)|x|^{\alpha}$$ be the $\alpha$-Hermite
operator for the hydrogen atom located at the origin in\, $ \mathbb
R^d,\, d\geq 2$, and $\Omega\subset \mathbb R^{d}$ be an open set.
The $\alpha$-Hermite operator is self-adjoint on the set of
infinitely differentiable functions with compact support, and it can
be factorized as
 \begin{equation*}\label{a0}
   \mathcal{H}_{\alpha}=\frac{1}{2}\sum_{i=1}^d(A^{+}_{i,\alpha}A^{-}_{i,\alpha}+A^{-}_{i,\alpha}A^{+}_{i,\alpha}),
\end{equation*}
 where
 $$
   A^{+}_{i,\alpha}=\partial_{x_i}+\sqrt{\alpha-1}x_i|x|^{(\alpha-2)/2},\qquad A^{-}_{i,\alpha}=\partial_{x_i}-\sqrt{\alpha-1}x_i|x|^{(\alpha-2)/2},
   \quad i=1,2,\ldots,d. $$
When $\alpha=1$, $A^{+}_{i,\alpha}$ and $A^{-}_{i,\alpha}$ are
exactly the classical partial derivatives $\partial_{x_i}$.  In the analysis
associated to $\mathcal{H}_{\alpha}$, the operators
$A^{\pm}_{i,\alpha}, 1\leq|i|\leq d$, play the same role as $\partial_{x_i}$ in the Euclidean
analysis. Refer to \cite{Bongioanni,Harboure,
Stempak,Stinga,Thangavelu} for the case of $-\mathcal{H}_2$. We call
$A^{\pm}_{i,\alpha}, 1\leq i\leq d$, the generalized derivatives
associated to $\mathcal{H}_{\alpha}$ and denote by
$$\nabla_{\mathcal{H}_{\alpha}}=(A^{-}_{d,\alpha},\ldots, A^{-}_{1,\alpha}, A_{1,\alpha},\ldots,A_{d,\alpha})$$
the generalized gradient. The partial differential equations defined
by the generalized gradient associated with $-{\mathcal{H}_2}$ can
be found in \cite{Molahajloo,Sjogren,Thangavelu1,Wong}.

We use $\mathcal{BV}_{\mathcal{H}_{\alpha}}(\Omega)$ to represent
the class of all functions with $\alpha$-Hermite Bounded Variation
(in short, $\alpha$-HBV) on $\Omega$. In Section \ref{sec-2.1}, we
investigate some basic properties of
$\mathcal{BV}_{\mathcal{H}_{\alpha}}(\Omega)$, e.g., the lower
semicontinuity (Lemma \ref{lem1.1}), the completeness  (Lemma
\ref{lemma2.3}) and approximation via $C^{\infty}_{c}$-functions
(Theorem \ref{propos4}). Section \ref{sec-2.2} is devoted to the
perimeter $P_{\mathcal{H}_{\alpha}}(\cdot)$ induced by
$\mathcal{BV}_{\mathcal{H}_{\alpha}}(\Omega)$, see (\ref{eq-2.1})
below. In Theorem \ref{thm2.6}, we obtain a coarea formula for
$\alpha$-HBV functions. As an application, we deduce that the
Sobolev type inequality:
$$\|f\|_{L^{{d}/{(d-1)}}}\lesssim \| \nabla_{\mathcal{H}_{\alpha}} f\|( \mathbb{R}^d)\ \ \ \ \ \ \forall\ f\in \mathcal{BV}_{\mathcal{H}_{\alpha}}( \mathbb{R}^d) $$
is equivalent to the following isoperimetric inequality:
$$|E|^{1-{1}/{d}}\lesssim \| \nabla_{\mathcal{H}_{\alpha}}
1_E\|(\mathbb{R}^d),$$ where $E$ is a bounded set with finite
$\alpha$-Hermite perimeter in $ \mathbb{R}^d$.

Recall that an elementary property of $P_{\mathcal{H}_{1}}(\cdot)$ is
\begin{equation}\label{eq-1.2}
P_{\mathcal{H}_{1}}(E)=P_{\mathcal{H}_{1}}(E^{c})\ \ \ \ \forall\ E
\subseteq \mathbb{R}^d.
\end{equation}
Unfortunately, we point out that, even for the convex set $E$, (\ref{eq-1.2}) is invalid for the general Hermite perimeter $P_{\mathcal{H}_{\alpha}}(\cdot)$.
 By the aid of Corollary \ref{coro2.5}, we construct a counterexample to show that there exists a convex set $E$ such that
  $P_{\mathcal{H}_{\alpha}}(E^c)=\infty$ while $P_{\mathcal{H}_{\alpha}}(E)<\infty$ (see (\ref{eq-converse-example})). In order to cover this shortage of $P_{\mathcal{H}_{\alpha}}(\cdot)$, we introduce a restricted version
 $\widetilde{P}_{\mathcal{H}_{\alpha}}(\cdot)$ such that the identity (\ref{eq-1.2}) holds true, see Lemma \ref{lem-5.1}.

In Section \ref{sec-3}, we introduce  the $\alpha$-HBV capacity
denoted by $\mathrm{cap}(E,
\mathcal{BV}_{\mathcal{H}_{\alpha}}(\mathbb R^{d}))$ for a set
$E\subseteq \mathbb R^{d}$. In Section \ref{sec-3.1}, we investigate
the measure-theoretic nature of $\mathrm{cap}(\cdot,
\mathcal{BV}_{\mathcal{H}_{\alpha}}(\mathbb R^{d}))$. Theorem
\ref{thm8} indicates that $\mathrm{cap}(\cdot, \mathcal{BV}(\mathbb
R^{d}))$ is not only an outer measure (obeying (i), (ii)\ \&\ (iv)),
but also a Choquet capacity (satisfying (i), (ii), (v)\ \&\ (vi)).
Denote by $[\mathcal{BV}_{\mathcal{H}_{\alpha}}(\mathbb
R^{d})]^{\ast}$ the dual space of
$\mathcal{BV}_{\mathcal{H}_{\alpha}}(\mathbb R^{d})$. In Section
\ref{s4}, we prove that a nonnegative Radon measure $\mu$ satisfying
one of the following two conditions:
$$\begin{cases}
\left|\int_{\mathbb R^d}f\,d\mu\right|\lesssim
\big(\|f\|_{L^1}+\parallel \nabla_{\mathcal{H}_{\alpha}}
f\parallel({\mathbb R^d}) \big) \quad\forall\, f\in
\mathcal{BV}_{\mathcal{H}_{\alpha}}(\mathbb R^d);\\
\mu(B)\lesssim  {\mathrm{cap}}(B,
\mathcal{BV}_{\mathcal{H}_{\alpha}}(\mathbb R^d))\ \forall\
\hbox{Borel\ set}\, B\subseteq \mathbb R^d.
\end{cases}$$
can be treated as a member of
$[\mathcal{BV}_{\mathcal{H}_{\alpha}}(\mathbb R^{d})]^{\ast}$.
Moreover, the above result derives a dual definition of
$\mathrm{cap}(\cdot, \mathcal{BV}_{\mathcal{H}_{\alpha}})$, see
Lemma \ref{l31} and Theorem \ref{t31}, respectively. Section
\ref{s5} is devoted to the trace and $\alpha$-HBV isocapacity
inequalities in $\mathbb R^d$. In Theorem \ref{thm18}, we obtain the
trace/restriction theorem arising from the end-point
$\alpha$-Hermite Sobolev space
$W^{1,1}_{\mathcal{H}_{\alpha}}(\mathbb R^{d})$. Further, assuming
that $\mu$ is a Lebesgue measure in Theorem \ref{thm18}, we derive
an imbedding result for the $\alpha$-Hermite operator
$\mathcal{H}_\alpha$. Let
$$\mathfrak C(f):=\Big\{\int^{\infty}_0
\Big[\mathrm{cap}(\{x\in  \mathbb R^d: |f(x)|\ge t\},
\mathcal{BV}_{\mathcal{H}_{\alpha}}(\mathbb
R^d))\Big]^{{{d}/{(d-1)}}}dt^{{{d}/({d-1})}}\Big\}^{{{(d-1)}/{d}}}.$$
In Theorem \ref{thm1}, we establish the following two equivalent
relations:
\begin{itemize}
\item[(i)] For any compactly supported $L^{{d}/{(d-1)}}$-function $f$, the analytic inequality:
$$\|f\|_{d/(d-1)}\lesssim \mathfrak C(f) \Longleftrightarrow |M|^{{{(d-1)}/{d}}}\lesssim
\mathrm{cap}(M, \mathcal{BV}_{\mathcal{H}_{\alpha}}(\mathbb R^d)),$$
where $M$ is any compact set  in $ \mathbb R^d$.
\item[(ii)] For any $f\in C^\infty_c(\mathbb R^d)$,
$$\mathfrak C(f)\lesssim  \|f\|_{\mathcal{BV}_{\mathcal{H}_{\alpha}}(\mathbb R^d)}\Longleftrightarrow\mathrm{cap}(M, \mathcal{BV}_{\mathcal{H}_{\alpha}}
(\mathbb R^d))\lesssim |M|+ P_{\mathcal{H}_{\alpha}}(M),$$ where $M$
is any connected compact set  in $\mathbb R^d$ with smooth boundary,
and  $P_{\mathcal{H}_{\alpha}}(M)$ is the $\alpha$-Hermite perimeter
of $M$.
\end{itemize}

In Section \ref{sec-5}, we want to investigate the $\alpha$-Hermite
mean curvature of a set with finite $\alpha$-Hermite perimeter. For
the special case, i.e., the Laplace operator $\mathcal{H}_{1}$, sets
of finite perimeter were introduced by E. De Giorgi in the 1950s,
and were applied to
 the research on some classical problems of the calculus of variations, such as the Plateau problem and the isoperimetric problem, see \cite{GCP},
 \cite{Giu} and \cite{Mas}. Barozzi-Gonzalez-Tamanini \cite{Barozzi} proved that every set $E$ of finite perimeter $P_{\mathcal{H}_{1}}(E)$
  in $\mathbb R^{d}$ has mean curvature in $L^1(\mathbb R^n)$. A natural question is that if the result of \cite{Barozzi} holds
  for $P_{\mathcal{H}_{\alpha}}(E), \alpha\in (1,\infty)$. We point out that, in the proof of main theorem of \cite{Barozzi},
  the identity (\ref{eq-1.2}) is required. The counterexample (\ref{eq-converse-example}) and Lemma \ref{lem-5.1} reveal that the restricted
  $\alpha$-Hermite perimeter, $\widetilde{P}_{\mathcal{H}_{\alpha}}(\cdot)$, is an appropriate substitute for $P_{\mathcal{H}_{1}}(\cdot)$ in
  general Hermite settings.
 In Theorem \ref{thm5-1}, we generalize the result of \cite{Barozzi} to $\widetilde{P}_{\mathcal{H}_{\alpha}}$ and prove that
 every set $E$ with $\widetilde{P}_{\mathcal{H}_{\alpha}}(E)<\infty$
  in $\mathbb R^{d}$ has mean curvature in $L^{1}(\mathbb R^{d})$. For the special case $\alpha=1$, Theorem \ref{thm5-1} coincides with \cite[page 314, Theorem]{Barozzi}, see Remark \ref{remark-5.1}.

\bigskip

\noindent{\it Some notations}:
\begin{itemize}
\item ${\mathsf U}\approx{\mathsf V}$ indicates that
there is a constant $c>0$ such that $c^{-1}{\mathsf V}\le{\mathsf
U}\le c{\mathsf V}$, whose right inequality is also written as
${\mathsf U}\lesssim{\mathsf V}$. Similarly, one writes ${\mathsf
V}\gtrsim{\mathsf U}$ for ${\mathsf V}\ge c{\mathsf U}$.

\item For convenience, the positive constant $C$
may change from one line to another and this usually depends on the spatial dimension $d$, the indices $p$,
and other fixed parameters.

\item Throughout this article, we use ${C}(\mathbb{R}^{d})$ to denote the spaces of all continuous functions on $\mathbb{R}^{d}$.
 Let $k\in\mathbb{N}\cup \{\infty\}$. The symbol  ${C}^{k}(\mathbb{R}^{d})$ denotes the class of all functions
  $f:\ \mathbb{R}^{d}\rightarrow \mathbb{R}$ with $k$ continuous partial derivatives. Denote by ${C}^{k}_{c}(\mathbb{R}^{d})$
  the class of all functions $f\in {C}^{k}(\mathbb{R}^{d})$ with compact
  support. The symbol $C^k(\mathbb R^d;\mathbb R^{2d})$ denotes the
  class of the functions $\varphi:\ \mathbb{R}^{d}\rightarrow \mathbb{R}^{2d}$, $\varphi=(\varphi_1,\varphi_2,\ldots,\varphi_{2d})$
  with $\varphi_i\in {C}^{k}(\mathbb{R}^{d})$ for $i=1,2,\ldots,
  2d$. The symbol $C^k_c(\Omega;\mathbb R^{2d})$ denotes the
  class of the functions $\varphi:\ \Omega\rightarrow \mathbb{R}^{2d}$, $\varphi=(\varphi_1,\varphi_2,\ldots,\varphi_{2d})$
  with $\varphi_i\in {C}^{k}(\Omega)$ for $i=1,2,\ldots,
  2d$.
\end{itemize}

%It follows from \cite{Shen} that the  operator
%$-\mathcal{H}_{\alpha}$ is also a Schr\"{o}dinger operator with the
%potential $$V(x)=|x|^{\alpha}\in B_q\ \ \text{for all}\ \ 1<q\le
%\infty,
%$$
%where $B_q$ is a class of locally $L^q$ integrable functions
%satisfying    the reverse H\"{o}lder inequality
%\begin{equation*}\label{e1} \left ( \frac{1}{|B|} \int_B V(x)^q
%\, dx \right )^{{1}/{q}} \lesssim  \frac{1}{|B|} \int_B V(x) \,
%dx
% \end{equation*}
% for every ball $B$ in $\mathbb{R}^d$ with $d\geq 3$. Assume $V\in
%B_{q_{_1}}$ for some $q_{_1}>{d}/{2}$. Then the auxiliary function
%$$\rho(x, V)= \rho(x)$$   introduced by Shen in \cite{Shen} is
%defined to be
%\begin{eqnarray*}
%\rho(x)=\frac{1}{m(x,V)}\dot{=} \sup_{r>0}\, \bigg \{ r:\;
%\frac{1}{r^{d-2}}\int_{B(x, r)}V(y)\, dy \leq 1 \bigg \}, \qquad
%x\in \mathbb{R}^d,
%\end{eqnarray*}
%which plays an important roles in the research of Schr\"{o}dinger
%operators.

\section{$\alpha$-HBV functions}\label{sec-2}
\subsection{Basic properties of $\mathcal{BV}_{\mathcal{H}_{\alpha}}(\Omega)$}\label{sec-2.1}
The divergence of a vector valued
function $$\varphi=(\varphi_1,\varphi_2,\ldots,\varphi_{2d})\in
C^1(\mathbb R^d;\mathbb R^{2d})$$ is
$$\mathrm{div}_{\mathcal{H}_{\alpha}}\varphi=A^{+}_{d,\alpha}\varphi_1+\cdots +A^{+}_{1,\alpha} \varphi_d+A^{-}_{1,\alpha} \varphi_{d+1}+\cdots
+A^{-}_{d,\alpha}\varphi_{2d}.$$  By a simple computation, we have
$$\mathrm{div}_{\mathcal{H}_{\alpha}}(\nabla_{\mathcal{H}_{\alpha}}
u)=2(\Delta-(\alpha-1)|x|^{\alpha})u.$$
Let $\Omega\subseteq \mathbb{R}^d$ be an open set.
  The  $\alpha$-Hermite variation of $f\in  {L}^1(\Omega)$ is defined by
\begin{equation*}
\parallel \nabla_{\mathcal{H}_{\alpha}} f\parallel(\Omega)=\sup_{\varphi\in
\mathcal{F}}\Big\{\int_{\Omega}f(x)
\mathrm{div}_{\mathcal{H}_{\alpha}}\varphi (x)dx\Big\},
\end{equation*}
where   $\mathcal{F}(\Omega)$ denotes the class of all functions
$$\varphi=(\varphi_1,\varphi_2,\ldots,\varphi_{2d})\in
C^1_c(\Omega;\mathbb R^{2d})$$ satisfying $$\parallel
\varphi\parallel_{\infty}=\sup_{x\in \Omega}(\mid
\varphi_1(x)\mid^2+\cdots+\mid \varphi_{2d}(x)\mid^2)^{1/2}\le 1.$$
An $L^1$ function $f$ is said to have the  $\alpha$-Hermite bounded
variation on $\Omega$ if $$\parallel \nabla_{\mathcal{H}_{\alpha}}
f\parallel(\Omega) < \infty,$$ and the collection of all such
functions is denoted by
$\mathcal{BV}_{\mathcal{H}_{\alpha}}(\Omega)$, which  is a Banach
space with the norm
$$\parallel f\parallel_{\mathcal{BV}_{\mathcal{H}_{\alpha}}(\Omega)} = \parallel
f\parallel_{L^1}+\parallel \nabla_{\mathcal{H}_{\alpha}}
f\parallel(\Omega),$$  see Lemma \ref{lemma2.3} below.   A
function $f\in L^1_{\loc}(\Omega,\mathbb{R})$ is said to be of
locally Hermite variation and we write $f\in
\mathcal{BV}_{H,\loc}(\Omega)$ if
\[
\parallel \nabla_{\mathcal{H}_{\alpha}} f\parallel(U)<\infty % \ \mbox{with}\ \overline{U}\ \mbox{compact}.
\]
holds true for every open set $U\subset \Omega$.

    Let $\Omega \subset \R^d$ be a  bounded open set and $E\subset \Omega$ be a Borel set.
    Then using \cite[Theorem 1.38]{EG}, it is easy to check that
    \begin{align*}
    \parallel \nabla_{\mathcal{H}_{\alpha}} f\parallel(E):=\inf\Big\{ \parallel \nabla_{\mathcal{H}_{\alpha}} f\parallel(U):\  E\subset U,\ U\subset\Omega\ \textrm{open}\Big\}
    \end{align*}
    extends $\parallel\nabla_{\mathcal{H}_{\alpha}} f\parallel(\cdot)$ to a Radon measure in $\Omega$.

In what follows,  we will collect   some properties of the space
$\mathcal{BV}_{\mathcal{H}_{\alpha}}(\Omega)$.  In \cite{Bongioanni}, the authors investigated the
  Sobolev spaces associated with $-\mathcal{H}_2$. The Sobolev spaces associated with $\mathcal{H}_{\alpha}$ in $\mathbb{R}^d$ can be similarly
  defined  as
  follows and they have same properties as the case of $-\mathcal{H}_2$.
\begin{definition}\label{defn1}Suppose $\Omega$ is an open set in
$\mathbb{R}^d$.
 Let $1\leq q<\infty$. The $\alpha$-Hermite Sobolev space $W^{1,q}_{\mathcal{H}_{\alpha}}(\Omega)$ associated with $\mathcal{H}_{\alpha}$ is
 defined as the set of all functions $f\in L^q(\Omega)$ such that
 $$A^{\pm}_{j,\alpha} f\in L^q(\Omega),\qquad 1\leq j\leq d.$$
 The norm of $f\in W^{1,q}_{\mathcal{H}_{\alpha}}(\Omega)$ is defined as
 $$\|f\|_{W^{1,q}_{\mathcal{H}_{\alpha}}}:=\sum_{1\leq j\leq d}\|A^{+}_{j,\alpha} f\|_{L^q}+\sum_{1\leq j\leq d}\|A^{-}_{j,\alpha} f\|_{L^q}+\|f\|_{L^q}. $$
 \end{definition}

\begin{lemma}\label{lem1.1}
\item{{\rm (i)}}   If $f\in W^{1,1}_{\mathcal{H}_{\alpha}}(\Omega)$, then
\begin{equation}\label{eq1.1}\parallel \nabla_{\mathcal{H}_{\alpha}} f \parallel(\Omega)=\int_{\Omega} |\nabla_{\mathcal{H}_{\alpha}} f(x)|dx.\end{equation}
% Moreover, if $f\in
%\mathcal{BV}_{\mathcal{H}_{\alpha}}(\Omega)\bigcap C^\infty(\Omega)$, then
%$f\in W^{1,1}_{\mathcal{H}_{\alpha}}(\Omega).$

\item{{\rm(ii)}} The $\alpha$-Hermite variation has the following lower semicontinuity:
if $$f,f_k\in \mathcal{BV}_{\mathcal{H}_{\alpha}}(\Omega), k\in
\mathbb{N},\ \ \text{ satisfy}\ \ f_k\rightarrow f \ \ \text{in}\ \
L^1_{\mathrm{loc}}(\Omega),
$$ then
\begin{equation}\label{equation1}\lim\inf_{k\rightarrow\infty}\parallel \nabla_{\mathcal{H}_{\alpha}}
f_k\parallel(\Omega)\ge \parallel \nabla_{\mathcal{H}_{\alpha}}
f\parallel(\Omega).
\end{equation}
\end{lemma}

\begin{proof}(i) If $f\in W^{1,1}_{\mathcal{H}_{\alpha}}(\Omega)$, then
$\nabla_{\mathcal{H}_{\alpha}} f\in L^1(\Omega)$. For every $\varphi\in C_c^{1}(\Omega; \mathbb R^{2d})$ with $\|\varphi\|_{L^\infty(\Omega)}\leq 1$,
we have
\begin{align*}
%\label{firstRe}
\left|\int_{\Omega}  f(x) \dvg_{\mathcal{H}_{\alpha}} \varphi(x)  dx\right|
& = \left|\int_{\Omega} \nabla_{\mathcal{H}_{\alpha}} f(x)\cdot \varphi(x)  dx\right|\leq \int_{\Omega}|\nabla_{\mathcal{H}_{\alpha}} f(x)|dx.
\end{align*}
    By taking the supremum over $\varphi$, we
    conclude that  $f\in  \mathcal{BV}_{\mathcal{H}_{\alpha}}(\Omega)$ and
    \begin{align}\label{pp}
   \parallel \nabla_{\mathcal{H}_{\alpha}} f \parallel(\Omega)\leq  \int_{\Omega}|\nabla_{\mathcal{H}_{\alpha}} f(x)| dx.
    \end{align}
    %\todo[inline]{Come segue la disuguaglianza opposta?}
  Define $\varphi\in L^\infty(\Omega;\mathbb R^{2d})$ as follows:
    $$
   \varphi(x):=
    \begin{cases}
     \frac{\nabla_{\mathcal{H}_{\alpha}} f(x)}{|\nabla_{\mathcal{H}_{\alpha}} f(x)|}, &
    \text{if $x\in\Omega$ and $\nabla_{\mathcal{H}_{\alpha}} f(x)\neq 0$,} \\
    0, & \text{otherwise.}
    \end{cases}
    $$
It is easy to see  that $\|\varphi\|_\infty \leq 1$.\ We choose a
sequence
    $\{\varphi_n\}_{n\in\N} \subset C^\infty_c(\Omega; \R^{2d})$ such that
    $\varphi_n\to \varphi$   as $n\to\infty$,
    with $\|\varphi_n\|_{L^\infty(\Omega)} \leq 1$ for all $n \in \N$.\
Combining the definition of $\parallel \nabla_{\mathcal{H}_{\alpha}} f \parallel(\Omega)$ with integration by parts derives that for every $n\geq 1$,
\begin{eqnarray*}
\parallel \nabla_{\mathcal{H}_{\alpha}} f \parallel(\Omega) & \geq& \sum_{i=1}^{d}\int_{\Omega} \big(\partial_{x_i}f(x)
 +\sqrt{\alpha-1}x_i|x|^{(\alpha-2)/2} f(x) \big)\varphi^{(d-i+1)}_n(x)dx\\
  &&+\sum_{i=1}^{d}\int_{\Omega} \big(\partial_{x_i}f(x)
 - \sqrt{\alpha-1}x_i|x|^{(\alpha-2)/2} f(x) \big)\varphi^{(i+d)}_n(x)dx\\
 &=& \int_{\Omega}  \nabla_{\mathcal{H}_{\alpha}} f(x)\cdot \varphi_n(x) dx.
\end{eqnarray*}
    By the dominated convergence theorem and the definition of $\varphi$,
    we have
    $$
    \parallel \nabla_{\mathcal{H}_{\alpha}} f \parallel(\Omega)\geq\int_{\Omega} |\nabla_{\mathcal{H}_{\alpha}} f(x)|dx
    $$ via letting $n\to \infty$,
    which is the opposite of inequality \eqref{pp}.

  % Suppose $f\in \mathcal{BV}_{\mathcal{H}_{\alpha}}(\Omega)\cap C^\infty(\Omega)$ and fix a compact set $K\subset\Omega$ with nonempty
%   interior. Consider
%    $
%    \tilde \varphi:=\varphi\chi_{{\rm int}(K)}.
%    $
%    Similarly, we choose a  sequence   $\{\varphi_n\}_{n\in\N}  \subset C^\infty_c({\rm int}(K),\R^{2d})$ such that $\varphi_n\to  \tilde\varphi$  pointwise
%     and $\|\varphi_n\|_{L^\infty({\rm int}(K))}\leq 1$, for all $n \in \N$. Then, we have
%    \begin{align*}
%    \parallel \nabla_{\mathcal{H}_{\alpha}} f \parallel(\Omega)& \geq \int_{\Omega}  f(x) \dvg_{\mathcal{H}_{\alpha}} \varphi_n(x) dx \\
%    &=\int_{K}     f(x) \dvg_{\mathcal{H}_{\alpha}} \varphi_n(x) dx \\
%    &= \sum_{i=1}^{d}\int_{K} \big(-\partial_{x_i}f(x)
%    + x_i f(x) \big)\varphi^{(i)}_n(x)dx+\sum_{i=1}^{d}\int_{K} \big(\partial_{x_i}f(x)
%    + x_i f(x) \big)\varphi^{(i+d)}_n(x)dx\\
%    &= \int_{K}  \nabla_{\mathcal{H}_{\alpha}} f(x)\cdot \varphi_n(x) dx.
%    \end{align*}
%    Since $f\in C^\infty(\Omega)$, we have
%    $\nabla_{\mathcal{H}_{\alpha}} f\in L^1(K)$. Thus, by the dominated convergence theorem,
%    $$
%     \parallel \nabla_{\mathcal{H}_{\alpha}} f \parallel(\Omega)\geq \int_{K} |\nabla\Re u(x)
%    + A(x)\Im u(x)|dx.
%    $$
%    The conclusion follows using an exhaustive sequence of compacts
%    via monotone convergence.

(ii) Fix $\varphi\in C^\infty_c(\Omega; \R^{2d})$ with
$\|\varphi\|_{L^\infty(\Omega)}\leq 1$. It follows from the definition  of
$\parallel \nabla_{\mathcal{H}_{\alpha}} f_k \parallel(\Omega)$ that
\begin{align*}
& \parallel \nabla_{\mathcal{H}_{\alpha}} f_{k}
\parallel(\Omega)\geq \int_{\Omega} f_k(x)
\dvg_{\mathcal{H}_{\alpha}}\varphi(x)  dx.
\end{align*}
By the convergence of $\{f_k\}_{k\in\N}$ in $L^1_{\loc}(\mathbb
R^d)$ to $f$ and Fatou's lemma, we get
\begin{align*}
& \liminf_{k\to\infty}\parallel \nabla_{\mathcal{H}_{\alpha}} f_k
\parallel(\mathbb R^d)\geq \int_{\Omega} f(x) \dvg_{\mathcal{H}_{\alpha}}\varphi(x)
dx.
\end{align*}
Therefore, (ii) can be proved  by the definition of $ \parallel
\nabla_{\mathcal{H}_{\alpha}} f \parallel(\Omega)$ and the
arbitrariness of such functions $\varphi$.
\end{proof}

The following lemma gives the  structure theorem   for
$\alpha$-Hermite BV functions.
\begin{lemma}\label{Struttura}
Let $\Omega \subset \R^d$ be a  bounded open set.  There exists a
unique $\R^{2d}$-valued finite Radon measure
$\mu_{\mathcal{H}_{\alpha},u}$ such that
\begin{equation*}
\begin{aligned}
\int_{\Omega}u(x) \dvg_{\mathcal{H}_{\alpha}} \varphi(x)  dx =
\int_{\Omega} \varphi(x)\cdot d \mu_{\mathcal{H}_{\alpha},u}  (x)
\end{aligned}
\end{equation*}
for every $\varphi\in C^{\infty}_c(\Omega; \R^{2d})$ and
\[
\parallel
\nabla_{\mathcal{H}_{\alpha}}
u\parallel(\Omega)=|\mu_{\mathcal{H}_{\alpha},u}|(\Omega).
\]
\end{lemma}
\begin{proof}
At first, we have
\[
\left|\int_{\Omega}u(x) \dvg_{\mathcal{H}_{\alpha}} \varphi(x)
dx\right|\leq \big(\parallel \nabla_{\mathcal{H}_{\alpha}}
u\parallel(\Omega)\big)\|\varphi\|_{L^\infty(\Omega)}\ \forall\
\varphi\in C^{\infty}_c(\Omega; \R^{2d}).
\]
Then, we use the Hahn-Banach theorem to conclude that there exists
  a linear and continuous extension $L$ of the functional
$\Psi:C^{\infty}_c(\Omega; \R^{2d})\to\mathbb R$ with
\[
\langle\Psi,\ \varphi\rangle= \int_{\Omega}u(x)
\dvg_{\mathcal{H}_{\alpha}} \varphi(x) \, dx
\]
to the normed space $(C_c(\Omega;
\R^{2d}),\|\cdot\|_{L^\infty(\Omega)})$ such that
$$
\|L\|=\|\Psi\|= \parallel \nabla_{\mathcal{H}_{\alpha}}
u\parallel(\Omega).
$$
%$\overline{L}_{1}(\varphi) : C_{c}(\Omega, \R^N) \to \R,$ such that
%\[\|\overline{L}_1\|:=\sup \left\{ \overline{L}_{1}(\varphi) | \varphi \in C_{c}(\Omega, \R^N), |\varphi| \leq 1, \textrm{supp}(\varphi) \subset K\right\} < +\infty,\]
%for every compact set $K \subset \Omega$.
Secondly, using the Riesz representation theorem (cf.\
\cite[Corollary 1.55]{ABF}),  we know that there exists a unique
$\R^{2d}$-valued finite Radon measure $\mu_{\mathcal{H}_{\alpha},
u}$ such that
\begin{align*}
%\label{intparts}
L (\varphi) = \int_{\Omega}\varphi(x) \cdot
d\mu_{\mathcal{H}_{\alpha},u}(x) \ \ \   \ \forall\  \varphi\in
C_c(\Omega;　\R^{2d})
\end{align*}
and such that $|\mu_{\mathcal{H}_{\alpha},u}|(\Omega)=\|L\|$. Thus
$|\mu_{\mathcal{H}_{\alpha},u}|(\Omega)=\parallel
\nabla_{\mathcal{H}_{\alpha}} u\parallel(\Omega)$.

\end{proof}

\begin{lemma}\label{lemma2.3}
The space
$(\mathcal{BV}_{\mathcal{H}_{\alpha}}(\Omega),\|\cdot\|_{\mathcal{BV}_{\mathcal{H}_{\alpha}}(\Omega)})$
is a Banach space.
\end{lemma}
\begin{proof}
    It is easy to check  that $\|\cdot\|_{\mathcal{BV}_{\mathcal{H}_{\alpha}}(\Omega)}$ is a norm. In what follows, we need to  prove that the space is complete.
    Let $\{f_n\}_{n\in\N}\subset \mathcal{BV}_{\mathcal{H}_{\alpha}}(\Omega)$ be a Cauchy sequence, namely, for every $\eps>0$ there exists $n_0\in\N$ such that
    $$\parallel\nabla_{\mathcal{H}_{\alpha}} (f_k-f_n) \parallel(\Omega)<\eps \quad
    \forall \ n,k\geq n_0.    $$
    In particular, $\{f_n\}_{n\in\N}$ is a Cauchy sequence in the Banach space
    $L^1(\mathbb
R^d)$, which implies that there exists $f\in L^1(\Omega)$ with
$\|f_n-f\|_{L^1(\Omega)}\to 0$  as $n\to\infty$. Then by  Lemma
\ref{lem1.1}, we have
    $$
    \parallel\nabla_{\mathcal{H}_{\alpha}} (f-f_k) \parallel(\Omega)\leq \liminf_n  \parallel\nabla_{\mathcal{H}_{\alpha}}
     (f_k-f_n) \parallel(\Omega)\leq\eps \quad
    \text{$\forall\,k\geq n_0$.}
    $$
So $ \parallel\nabla_{\mathcal{H}_{\alpha}} (f_k-f)
\parallel(\Omega) \to 0$ as $k\to\infty$. This completes the proof.
    \end{proof}

Next we will list the following approximation result  for the
$\alpha$-Hermite variation.

\begin{theorem}\label{propos4}
If $u\in \mathcal{BV}_{\mathcal{H}_{\alpha}}(\Omega)$, there exists a
sequence of functions $\{u_{h}\}_{h\in\mathbb{N}}\in
C^\infty_c(\Omega)\cap \mathcal{BV}_{\mathcal{H}_{\alpha}}(\Omega)$
such that\,　$ \lim\limits_{h\rightarrow\infty}\parallel u_{h}-u\parallel_{L^1}=0$ and
$$\lim_{h\rightarrow\infty}\int_{\Omega}
|\nabla_{\mathcal{H}_{\alpha}} u_{h}(x)|dx=
\parallel \nabla_{\mathcal{H}_{\alpha}} u\parallel(\Omega).$$
\end{theorem}

\begin{proof}
We adopt the method similar to the proof of \cite[Theorem 5.3]{EG}.
Via the semicontinuity property of Lemma~\ref{lem1.1}, we only need
to verify that, for every $\eps
>0$, there exists a function $u_{\eps} \in C^{\infty}(\Omega)$ such
that
\begin{equation}\label{enough}
\int_{\Omega}|u(x)-u_\eps(x)| dx < \eps\ \quad \&\   \quad
\parallel \nabla_{\mathcal{H}_{\alpha}} u_\varepsilon\parallel(\Omega) < \parallel \nabla_{\mathcal{H}_{\alpha}} u\parallel(\Omega) + \eps.
\end{equation}

Given a positive integer $m$, let $\{\Omega_{j}\}_{j\in\N}$ be a
sequence of open sets  defined as
$$
\Omega_j := \Big\{ x \in \Omega \,\,|\,\, \textrm{dist}(x,\partial
\Omega) > \dfrac{1}{m+j}\Big\} \cap B(0,k+m), \qquad j \in\N,
$$
\noindent where $B(0,k+m)$ denotes the open ball of center $0$ and
radius $k+m$,   and   $\mathrm{dist}(x,\partial \Omega)$ represents
the Euclidean  distance from $x$ to $ \partial \Omega$.  Since
$\parallel \nabla_{\mathcal{H}_{\alpha}} u\parallel(\cdot)$ is a
Radon measure, given $\eps>0$ we can choose $m \in \N$ so large that
\begin{equation}\label{misura}
\parallel \nabla_{\mathcal{H}_{\alpha}} u\parallel(\Omega \setminus \Omega_0) < \eps.
\end{equation}
In fact, we find that the sequence of open sets $\{\Omega_{j}\}$
satisfy the following properties:
$$\left\{\begin{aligned}
&\Omega_{j} \subset \Omega_{j+1} \subset \Omega\quad  \forall\ j\in\mathbb N;\\
&\bigcup_{j=0}^\infty
\Omega_{j} = \Omega.
 \end{aligned}\right.$$

Set $ U_0 := \Omega_0$ and $U_j := \Omega_{j+1} \setminus
\overline{\Omega}_{j-1}$  for $j\geq 1$.  By
standard results (cf. \cite{EG}), there exists a partition of unity
related to the covering $\{U_j\}_{j\in\N}$, which means that there
exists $\{ f_{j} \}_{j\in\N} \in C^{\infty}_{c}(U_j)$ such that
$0\leq f_j \leq 1$ for every $j \geq 0$ and
$\sum\limits_{j=0}^\infty f_j = 1$ on $\Omega$.
 In particular, the following fact is valid:
\begin{equation}\label{partition}
\sum_{j=0}^{\infty} \nabla f_j = 0 \quad\,\, \textrm{on $\Omega$}.
\end{equation}

Let $\eta\in C^{\infty}_c(\R^d)$ be a radial nonnegative function
satisfying $\int_{\R^d} \eta(x) dx=1$ and ${\rm supp}(\eta)\subset
B(0,1)$. Given $\varepsilon>0$ and $u\in L^1(\Omega; \mathbb{R})$,
extended to zero out of $\Omega$, we define the usual regularization
\begin{align*}\label{mollif}
    u_\varepsilon(x):=\frac{1}{\varepsilon^d}\int_{\mathbb R^d}\eta\Big(\frac{x-y}{\varepsilon}\Big) u(y) dy=
    \frac{1}{\varepsilon^d}\int_{B(x,\eps)}\eta\Big(\frac{x-y}{\varepsilon}\Big) u(y) dy.
    %=\frac{1}{\varepsilon^n}\int_{B(0,\epsilon)}\eta\left(\frac{y}{\varepsilon}\right) u(x-y) dy.
\end{align*}

For every $j \geq 0$, there exists $0<\eps_j<\eps$ such that
\begin{equation}\label{eq-ast}
\left\{\begin{aligned}
\hspace{-0.5cm}\mathrm{supp}\big((f_ju)_{\varepsilon_j}\big)\subseteq U_j;　\\
\int_{\Omega}\mid (f_ju)_{\varepsilon_j}-f_ju \mid dx<\varepsilon 2^{-(j+1)};　\ \ \ \\
\hspace{-3.5cm}\int_{\Omega}\mid (u\nabla
f_j)_{\varepsilon_j}-u\nabla f_j \mid dx<\varepsilon 2^{-(j+1)}.
 \end{aligned}\right.
 \end{equation}
Define $\phi_\varepsilon:=\sum\limits^\infty_{j=0}(uf_j)_{\varepsilon_j}$.
Since the sum is locally finite, then we conclude that
$\phi_\varepsilon\in C^\infty(\Omega)$ and $u=\sum\limits^\infty_{j=0}uf_j$
pointwise. By a direct computation, we can get
\begin{eqnarray*}
&&\int_{\Omega}\phi_\varepsilon(x)
\mathrm{div}_{\mathcal{H}_{\alpha}}
\varphi(x)dx\\
&&=\sum^\infty_{j=0}\int_{\Omega} \big((uf_j)\ast
\eta_{\varepsilon_j}\big)(x)\mathrm{div}_{\mathcal{H}_{\alpha}} \varphi(x)dx\\
&&=\sum^\infty_{j=0}\int_{\Omega}\int_{\Omega}\frac{1}{\varepsilon^d_j}
\eta\Big(\frac{x-y}{\varepsilon_j}\Big)
u(y)f_j(y) \mathrm{div}_{\mathcal{H}_{\alpha}} \varphi(x)dydx\\
&&=
\sum^\infty_{j=0}\int_{\Omega}\int_{\Omega}\frac{1}{\varepsilon^d_j}
\eta\Big(\frac{x-y}{\varepsilon_j}\Big) u(y)f_j(y)
\Big[A^{+}_d\varphi_1+\cdots+ A^{+}_1 \varphi_d+A^{-}_{1}
\varphi_{d+1}+\cdots+ A^{-}_{d}\varphi_{2d}\Big]dydx\\
&&:=I+II,
\end{eqnarray*}
where
$$I:=\sum^\infty_{j=0}\int_{\Omega}\int_{\Omega}\frac{1}{\varepsilon^d_j}
\eta\Big(\frac{x-y}{\varepsilon_j}\Big) u(y)f_j(y)
\Big[\sum^d_{k=1}\frac{\partial}{\partial
x_k}\varphi_{d-k+1}(x)+\sum^d_{k=1} \frac{\partial}{\partial
x_k}\varphi_{d+1+k}(x)\Big]dydx$$
and
\begin{eqnarray*}
II&:=&\sum^\infty_{j=0}\int_{\Omega}\int_{\Omega}\frac{1}{\varepsilon^d_j}
\eta\Big(\frac{x-y}{\varepsilon_j}\Big) u(y)f_j(y)
\Big[\sum^d_{k=1}
\sqrt{\alpha-1}x_k|x|^{(\alpha-2)/2}\varphi_{d-k+1}(x)\\
&&-\sum^d_{k=1}
 \sqrt{\alpha-1}x_k|x|^{(\alpha-2)/2}\varphi_{d+1+k}(x)\Big]dydx.
\end{eqnarray*}

As for $I$, let
$$I_{1}:=\sum^\infty_{j=0}\int_{\Omega}
 u(y)
 \Big[\sum^d_{k=1}\frac{\partial}{\partial
y_k}\Big(f_j(y)(\varphi_{d-k+1}\ast
\eta_{\varepsilon_j}(y))\Big)+\sum^d_{k=1} \frac{\partial}{\partial
y_k}\Big(f_j(y)(\varphi_{d+1+k}\ast
\eta_{\varepsilon_j}(y))\Big)\Big]dy$$
and
\begin{eqnarray*}
I_{2}&:=& -\sum^\infty_{j=0}\int_{\Omega}
 \Big\{\Big[\sum^d_{k=1}\big(u(y)\frac{\partial}{\partial
y_k}f_j(y)\big)\ast
\eta_{\varepsilon_j}(y))-u\frac{\partial}{\partial
y_k}f_j(y)\Big]\varphi_{d-k+1}(y)\\
 &&-\Big[\sum^d_{k=1}\big(u(y)
\frac{\partial}{\partial y_k}f_j(y)\big)\ast
\eta_{\varepsilon_j}(y))-u\frac{\partial}{\partial
y_k}f_j(y)\Big]\varphi_{d+1+k}(y)\Big\}dy.
\end{eqnarray*}
We can get
\begin{eqnarray*}
I&=&\sum^\infty_{j=0}\int_{\Omega}
 u(y)f_j(y)
 \big(\sum^d_{k=1}\frac{\partial}{\partial
y_k}(\varphi_{d-k+1}\ast \eta_{\varepsilon_j}(y))+\sum^d_{k=1}
\frac{\partial}{\partial y_k}(\varphi_{d+1+k}\ast
\eta_{\varepsilon_j}(y))\big)dy\\
&=& \sum^\infty_{j=0}\int_{\Omega}
 u(y)
 \big(\sum^d_{k=1}\frac{\partial}{\partial
y_k}\big(f_j(y)(\varphi_{d-k+1}\ast
\eta_{\varepsilon_j}(y))\big)+\sum^d_{k=1} \frac{\partial}{\partial
y_k}\big(f_j(y)(\varphi_{d+1+k}\ast
\eta_{\varepsilon_j}(y))\big)\big)dy\\
&& -\sum^\infty_{j=0}\int_{\Omega}
 u(y)
 \big(\sum^d_{k=1}\frac{\partial}{\partial
y_k}\big(f_j(y)\big)(\varphi_{d-k+1}\ast
\eta_{\varepsilon_j}(y))+\sum^d_{k=1} \frac{\partial}{\partial
y_k}\big(f_j(y)\big)(\varphi_{d+1+k}\ast
\eta_{\varepsilon_j}(y))\big)dy\\
&=& I_1+I_2,
\end{eqnarray*}
where we have used (\ref{partition}) in
the last equality. When $\|\varphi\|_{L^\infty}\le 1$, it holds that
$$\begin{cases}
|\big(f_j(y)\big)(\varphi_{d-k+1}\ast \eta_{\varepsilon_j}(y))|\leq
1,\\
|\big(f_j(y)\big)(\varphi_{d+k+1}\ast
\eta_{\varepsilon_j}(y))|\leq 1
\end{cases}$$
 for all $j\ge 0$ and
$k=1,2,\ldots,d$. Moreover,  it follows from (\ref{eq-ast}) that
$|I_2|<\varepsilon$.

For $II$, a direct computation gives
\begin{eqnarray*}
II&=&\sum^\infty_{j=0}\int_{\Omega}\int_{\Omega}\frac{1}{\varepsilon^d_j}
\eta\Big(\frac{x-y}{\varepsilon_j}\Big) u(y)f_j(y) (\sum^d_{k=1}
\sqrt{\alpha-1}y_k|y|^{(\alpha-2)/2}
\varphi_{d-k+1}(x)\\
&&-\sum^d_{k=1}
 \sqrt{\alpha-1}y_k|y|^{(\alpha-2)/2}\varphi_{d+1+k}(x))dydx\\
&&+\sum^\infty_{j=0}\int_{\Omega}\int_{\Omega}\frac{1}{\varepsilon^d_j}
\eta\Big(\frac{x-y}{\varepsilon_j}\Big) u(y)f_j(y) (\sum^d_{k=1}
\sqrt{\alpha-1}{
(x_k|x|^{(\alpha-2)/2}-y_k|y|^{(\alpha-2)/2}})\varphi_{d-k+1}(x)\\
&&-\sum^d_{k=1}
 \sqrt{\alpha-1}{
(x_k|x|^{(\alpha-2)/2}-y_k|y|^{(\alpha-2)/2}})\varphi_{d+1+k}(x))dydx.
\end{eqnarray*}
Changing the order of integration, we   get
\begin{eqnarray*}
II&=&\sum^\infty_{j=0}\int_{\Omega}  u(y) \Big[\sum^d_{k=1}
 \sqrt{\alpha-1}y_k|y|^{(\alpha-2)/2}f_j(y)(\varphi_{d-k+1}\ast
\eta_{\varepsilon_j}(y))\\
&&-\sum^d_{k=1}
 \sqrt{\alpha-1}y_k|y|^{(\alpha-2)/2}f_j(y)(\varphi_{d+k+1}\ast
\eta_{\varepsilon_j}(y))\Big]dy\\
&&+\sum^\infty_{j=0}\int_{\Omega}\int_{\Omega}\frac{1}{\varepsilon^d_j}
\eta\Big(\frac{x-y}{\varepsilon_j}\Big) u(y)f_j(y) \Big[\sum^d_{k=1}
\sqrt{\alpha-1}{
(x_k|x|^{(\alpha-2)/2}-y_k|y|^{(\alpha-2)/2}})\varphi_{d-k+1}(x)\\&&-\sum^d_{k=1}
 \sqrt{\alpha-1}{
(x_k|x|^{(\alpha-2)/2}-y_k|y|^{(\alpha-2)/2}})\varphi_{d+1+k}(x)\Big]dydx.
\end{eqnarray*}
 Therefore, the above estimate for the term  $I_{2}$ indicates that
$$\Big|\int_{\Omega}\phi_\varepsilon(x)
\mathrm{div}_{\mathcal{H}_{\alpha}} \varphi(x)dx\Big|=
|I_{1}+I_{2}+II|\leq J_{1}+J_{2}+\varepsilon,$$ where
\begin{eqnarray*}
J_{1}&:=&\Big|\Big[\sum^\infty_{j=0}\int_{\Omega}
 u(y)
 \big(\sum^d_{k=1}\frac{\partial}{\partial
y_k}\big(f_j(y)(\varphi_{d-k+1}\ast
\eta_{\varepsilon_j}(y))\big)+\sum^d_{k=1} \frac{\partial}{\partial
y_k}\big(f_j(y)(\varphi_{d+1+k}\ast
\eta_{\varepsilon_j}(y))\big)\big)dy\Big]\\
&&+\Big[\sum^\infty_{j=0}\int_{\Omega}  u(y) (\sum^d_{k=1}
\sqrt{\alpha-1}y_k|y|^{(\alpha-2)/2}f_j(y)(\varphi_{d-k+1}\ast
\eta_{\varepsilon_j}(y))\\
&&-\sum^d_{k=1}
 \sqrt{\alpha-1}y_k|y|^{(\alpha-2)/2}f_j(y)(\varphi_{d+k+1}\ast
\eta_{\varepsilon_j}(y)))dydx\Big]\Big|
\end{eqnarray*}
 and
\begin{eqnarray*}
J_{2}&:=&\Big|\sum^\infty_{j=0}\int_{\Omega}\int_{\Omega}\frac{1}{\varepsilon^d_j}
\eta\Big(\frac{x-y}{\varepsilon_j}\Big) u(y)f_j(y) \Big[\sum^d_{k=1}
\sqrt{\alpha-1}{
(x_k|x|^{(\alpha-2)/2}-y_k|y|^{(\alpha-2)/2}})\varphi_{d-k+1}(x)\\
&&-\sum^d_{k=1}
 \sqrt{\alpha-1}{
(x_k|x|^{(\alpha-2)/2}-y_k|y|^{(\alpha-2)/2}})\varphi_{d+1+k}(x)\Big]dydx\Big|.
\end{eqnarray*}

   Note that, by the construction of $U_j$, every
point $x\in\Omega$ belongs to at most three of the sets $U_j$.
Similar  to \cite[Section 5.2.2, Theorem 2]{EG}, we know that
\begin{eqnarray*}
J_1&\le&\Big|\Big\{\int_{\Omega}
 u(y)\Big[\sum^d_{k=1}\frac{\partial}{\partial
y_k}\Big(f_0(y)(\varphi_{d-k+1}\ast
\eta_{\varepsilon_0}(y))\Big)+\sum^d_{k=1} \frac{\partial}{\partial
y_k}\Big(f_0(y)(\varphi_{d+1+k}\ast
\eta_{\varepsilon_0}(y))\Big)\Big]dy\Big\}\\
&&-\Big\{\int_{\Omega}  u(y) \Big[\sum^d_{k=1}
\sqrt{\alpha-1}y_k|y|^{(\alpha-2)/2}f_0(y)(\varphi_{d-k+1}\ast
\eta_{\varepsilon_0}(y))\\
&&-\sum^d_{k=1}
 \sqrt{\alpha-1}y_k|y|^{(\alpha-2)/2}f_0(y)(\varphi_{d+k+1}\ast
\eta_{\varepsilon_0}(y))\Big]dydx\Big\}\Big|\\
&&+\Big|\Big\{\sum^\infty_{j=1}\int_{\Omega}
 u(y)
 \Big[\sum^d_{k=1}\frac{\partial}{\partial
y_k}\Big(f_j(y)(\varphi_{d-k+1}\ast
\eta_{\varepsilon_j}(y))\Big)+\sum^d_{k=1} \frac{\partial}{\partial
y_k}\Big(f_j(y)(\varphi_{d+1+k}\ast
\eta_{\varepsilon_j}(y))\Big)\Big]dy\Big\}\\
&&-\Big\{\sum^\infty_{j=1}\int_{\Omega}  u(y) \Big[\sum^d_{k=1}
\sqrt{\alpha-1}y_k|y|^{(\alpha-2)/2}f_j(y)(\varphi_{d-k+1}\ast
\eta_{\varepsilon_j}(y))\\
&&-\sum^d_{k=1}
 \sqrt{\alpha-1}y_k|y|^{(\alpha-2)/2}f_j(y)(\varphi_{d+k+1}\ast
\eta_{\varepsilon_j}(y))\Big]dydx\Big\}\Big|\\
&\lesssim&  \parallel \nabla_{\mathcal{H}_{\alpha}} u
\parallel(\Omega)+ \sum^\infty_{j=1}\parallel \nabla_{\mathcal{H}_{\alpha}} u
\parallel(U_j)\\
&\lesssim&  \parallel \nabla_{\mathcal{H}_{\alpha}} u
\parallel(\Omega)+3 \parallel \nabla_{\mathcal{H}_{\alpha}} u
\parallel(\Omega\backslash \Omega_0)\\
&\lesssim& \parallel \nabla_{\mathcal{H}_{\alpha}} u
\parallel(\Omega)+3 \varepsilon,
\end{eqnarray*}
where  we have used (\ref{misura}) in the last inequality.

Noting that $\psi(x)=x_k|x|^{(\alpha-2)/2}$ is Lipschitz
continuous, $\|\varphi\|\le 1$ and $\mathrm{supp}(\eta)\subseteq
B_1(0)$, then we have
\begin{eqnarray*}
J_2&\lesssim&  \varepsilon\mathrm{Lip}(\psi,\Omega)
\int_{\mathbb{R}^d}\eta(z)dz \int_{\Omega}
\sum^\infty_{j=1}|f_j(y)||u(y)|dy\lesssim  \varepsilon,
\end{eqnarray*}
where $\mathrm{Lip}(\psi,\Omega)$ denotes the Lipschitz constant of $\psi$.
By taking the supremum over $\varphi$ and the arbitrariness of
$\varepsilon>0$,  we conclude that (\ref{enough}) holds true.

\end{proof}

Moreover, we have the following max-min property of the
$\alpha$-Hermite variation.

\begin{theorem}\label{theorem2}
Let $u,v\in L^1(\Omega)$. Then
$$\parallel  \nabla_{\mathcal{H}_{\alpha}} \max\{u, v\}\parallel(\Omega)+\parallel  \nabla_{\mathcal{H}_{\alpha}} \min\{u, v\}\parallel(\Omega)\le \parallel  \nabla_{\mathcal{H}_{\alpha}}  u\parallel(\Omega)
+\parallel  \nabla_{\mathcal{H}_{\alpha}}  v\parallel(\Omega).$$
\end{theorem}
\begin{proof} Without loss of generality,  we may assume
$$
\parallel  \nabla_{\mathcal{H}_{\alpha}}  u\parallel(\Omega) +
\parallel  \nabla_{\mathcal{H}_{\alpha}}  v\parallel(\Omega)<\infty.
$$
Take two functions
$$u_{h}, v_{h}\in C^\infty_c(\Omega)\cap
\mathcal{BV}_{\mathcal{H}_{\alpha}}(\Omega), h=1,2,\ldots,$$  such that
$$
\begin{cases}
u_{h}\rightarrow
u, v_{h}\rightarrow v\quad\hbox{in}\quad L^1(\Omega);\\
\int_{\Omega} |\nabla_{\mathcal{H}_{\alpha}} u_{h}(x)|dx\rightarrow
\parallel  \nabla_{\mathcal{H}_{\alpha}}  u\parallel(\Omega);\\
\int_{\Omega} |\nabla_{\mathcal{H}_{\alpha}} v_{h}(x)|dx\rightarrow
\parallel  \nabla_{\mathcal{H}_{\alpha}}  v\parallel(\Omega).
\end{cases}
$$
Since
$$
\max\{u_{h},v_{h}\}\rightarrow \max\{u,v\}\quad\&\quad
\min\{u_{h},v_{h}\}\rightarrow \min\{u,v\}\quad\hbox{in}\quad
L^1(\Omega),
$$
it follows that

\begin{eqnarray*}
\parallel  \nabla_{\mathcal{H}_{\alpha}} \max\{u, v\}\parallel(\Omega)&+&\parallel  \nabla_{\mathcal{H}_{\alpha}} \min\{u, v\}\parallel(\Omega)\\
&\le&\lim\inf_{h\rightarrow\infty}\int_{\Omega}
|\nabla_{\mathcal{H}_{\alpha}}
\max\{u_{h},v_{h}\}|dx+\lim\inf_{h\rightarrow\infty}\int_{\Omega}
|\nabla_{\mathcal{H}_{\alpha}}
\min\{u_{h},v_{h}\}|dx\\
&\le& \lim\inf_{h\rightarrow\infty}\big(\int_{\Omega}
|\nabla_{\mathcal{H}_{\alpha}} \max\{u_{h},v_{h}\}|dx+ \int_{\Omega}
|\nabla_{\mathcal{H}_{\alpha}}
\min\{u_{h},v_{h}\}|dx\big)\\
&\le& \lim_{h\rightarrow\infty}\int_{\Omega}
|\nabla_{\mathcal{H}_{\alpha}}
u_{h}(x)|dx+\lim_{h\rightarrow\infty}\int_{\Omega}
|\nabla_{\mathcal{H}_{\alpha}}
v_{h}(x)|dx\\
&=&\parallel  \nabla_{\mathcal{H}_{\alpha}}  u\parallel(\Omega)
+\parallel \nabla_{\mathcal{H}_{\alpha}} v\parallel(\Omega).
\end{eqnarray*}

\end{proof}

\subsection{$\alpha$-Hermite perimeter}\label{sec-2.2}
In this subsection, we introduce two kinds of  new perimeters: the
$\alpha$-Hermite perimeter and the restricted $\alpha$-Hermite
perimeter.  We also establish related theories  for them.

The $\alpha$-Hermite perimeter of $E\subseteq \Omega$ can be defined
as follows:
\begin{equation}\label{eq-2.1}
P_{\mathcal{H}_{\alpha}}(E,\Omega)=\parallel \nabla_{\mathcal{H}_{\alpha}} 1_E \parallel(\Omega) =
\sup_{\varphi\in \mathcal {F}(\Omega)}\Big\{\int_E
\mathrm{div}_{\mathcal{H}_{\alpha}}\varphi(x)dx \Big\},
\end{equation}
  where
$1_E$ denotes the characteristic function of $E$.  It should be
noted that for $\alpha=1$
$$P_{\mathcal{H}_{\alpha}}(E,\Omega)=2P(E,\Omega),$$ where $P(E,\Omega)$ is
exactly the classical perimeter of $E\subseteq \Omega$. In
particular, we shall also write $$P_{\mathcal{H}_{\alpha}}(E,
\mathbb{R}^d)= P_{\mathcal{H}_{\alpha}}(E).$$

The following conclusion is a direct corollary of Lemma
\ref{lem1.1}.

\begin{corollary}\label{lem1.1-1}
The $\alpha$-Hermite perimeter has the following lower semicontinuity:
if $$  1_{E_k}\rightarrow 1_E \ \ \text{in}\ \
L^1_{\mathrm{loc}}(\Omega),
$$ where $E_k$ and $E$ are subsets of $\Omega$ for $k=1,2,\ldots$, then
\begin{equation}\label{equation1}\lim\inf_{k\rightarrow\infty}P_{\mathcal{H}_{\alpha}}(E_k,\Omega)\ge P_{\mathcal{H}_{\alpha}}(E,\Omega).
\end{equation}
\end{corollary}

For any compact subsets $E, F$ in $\Omega$, via choosing $u=1_E$ and
$v=1_F$, the  following lemma can be deduced from Theorem
\ref{theorem2} immediately.

\begin{lemma}\label{lem3.6} For any  subsets $E$ in $
\Omega$,
$$P_{\mathcal{H}_{\alpha}}(E\cap F, \Omega)+P_{\mathcal{H}_{\alpha}}(E\cup F, \Omega)\le P_{\mathcal{H}_{\alpha}}(E, \Omega)+
P_{\mathcal{H}_{\alpha}}( F, \Omega).$$
\end{lemma}

In what follows, we establish the coarea formula for $\alpha$-HBV
functions. Let $f: \Omega\rightarrow \mathbb{R}^d$ and $t\in
\mathbb{R}$. Denote by $E_t=\{x\in \Omega: f(x)>t\}.$ The structure
of the $\alpha$-Hermite divergence and \cite[Section 5.5, Lemma
1]{EG} imply the following lemma.
\begin{lemma}If $f\in \mathcal{BV}_{\mathcal{H}_{\alpha}}(\Omega)$, the mapping $t\mapsto P_{\mathcal{H}_{\alpha}}(E_t,\Omega)$ is Lebesgue measurable for  $ t\in \mathbb{R}$.
\end{lemma}
Below we prove a coarea formula for $\alpha$-HBV functions.
\begin{theorem}\label{thm2.6} If $f\in
\mathcal{BV}_{\mathcal{H}_{\alpha}}(\Omega)$, then
\begin{equation}\label{eq2.8}\parallel  \nabla_{\mathcal{H}_{\alpha}}  f\parallel(\Omega)
\approx\int^\infty_{-\infty} P_{\mathcal{H}_{\alpha}}(E_t,\Omega)dt.
\end{equation}
\end{theorem}
\begin{proof}Let $\varphi\in C^1_c(\Omega, \mathbb{R}^{2d})$ and $\|\varphi\|_{L^\infty}\le
1.$ Firstly, we claim that $$\int_{\Omega}f
\mathrm{div}_{\mathcal{H}_{\alpha}}\varphi
dx=\int^\infty_{-\infty}\Big(\int_{E_t}\mathrm{div}_{\mathcal{H}_{\alpha}}\varphi
dx\Big)dt.$$  The above claim can be proved by the following facts:
for $i=1,2,\ldots,d$,
$$\int_{\Omega}f x_i|x|^{(\alpha-2)/2}\varphi
dx=\int^\infty_{-\infty}\Big(\int_{E_t} x_i|x|^{(\alpha-2)/2}\varphi
dx\Big)dt $$ and
$$\int_{\Omega}f \mathrm{div} \varphi
dx=\int^\infty_{-\infty}\Big(\int_{E_t}\mathrm{div} \varphi
dx\Big)dt,$$ where the latter can be seen in the proof of
\cite[Section 5.5, Theorem 1]{EG}. Therefore, we conclude that for all
$\varphi$ as above,
 $$\int_{\Omega}f \mathrm{div}_{\mathcal{H}_{\alpha}}\varphi
dx\le \int^\infty_{-\infty}P_{\mathcal{H}_{\alpha}}(E_t,\Omega)
dt.$$ Furthermore,
 \begin{equation*}
 \|\nabla_{\mathcal{H}_{\alpha}} f \|(\Omega)\le \int^\infty_{-\infty}P_{\mathcal{H}_{\alpha}}(E_t,\Omega) dt.
\end{equation*}
Secondly, we claim that (\ref{eq2.8}) holds true for all $f\in
\mathcal{BV}_{\mathcal{H}_{\alpha}}(\Omega)\bigcap
C^\infty(\Omega)$. Next we will prove the claim according to the
idea of \cite[Proposition 4.2]{Miranda1}. Let
\begin{equation*}
\label{eqq2.8}m(t)=\int_{\{x\in\Omega: f(x)\le
t\}}|\nabla  f|dx.
\end{equation*}
Then it is obvious that
\begin{equation*}\label{eqq2.9}\int^\infty_{-\infty}m'(t)dt\le\int_{\Omega}|\nabla f|dx.\end{equation*}
Define the following function $g_{h}$:
$$g_{h}(s)=\begin{cases} 0, \mathrm{if} \ s\le t;\\
 h(t-s)+1, \ \mathrm{if} \ t\le s\le t+1/h;\\
1, \ \mathrm{if} \   s\ge t+1/h,
\end{cases}$$
where $t\in \mathbb{R}$.  We define the sequence
$v_{h}(x):=g_{h}(f(x))$. At this time, $v_{h}\rightarrow 1_{E_t}$ in
$L^1(\Omega)$. In fact,
\begin{eqnarray*}
\int_{\Omega}|v_{h}(x)-1_{E_t}(x)|dx&=&\int_{\{x\in\Omega:
t<f(x)\le t+1/h\}}g_{h}(f(x))dx\\
&\le& \Big|\Big\{x\in\Omega:
t<f(x)\le t+1/h\Big\}\Big|\rightarrow 0,
\end{eqnarray*}
since $\{x\in\Omega: t<f(x)\le t+1/h\}\rightarrow \emptyset$ when
$h\rightarrow \infty.$ By a simple computation and [17, (2.19)],  we
obtain
\begin{equation}\label{eq3.11}|\nabla f(x)|\le
|\nabla_{\mathcal{H}_{\alpha}} f(x)|\le\sqrt{2} (|\nabla
f(x)|+\sqrt{\alpha-1}|x|^{\alpha/2}|f(x)|).\end{equation}
 Then
\begin{eqnarray*} \int_{\Omega}|\nabla_{\mathcal{H}_{\alpha}} v_{h}(x)|dx&\le&
\sqrt{2}h\int_{\{x\in\Omega: t<f(x)\le t+1/h\}} |\nabla
f(x)|dx+2\sqrt{2(\alpha-1)}\int_{\{x\in\Omega: t<f(x)\le t+1/h\}} |x
|^{\alpha/2}dx\\
&&+\sqrt{2(\alpha-1)}\int_{\{x\in\Omega:  f(x)\ge t+1/h\}} |x
|^{\alpha/2}dx\\
&=&
\sqrt{2}h\big(m(t+1/h)-m(t)\big)+2\sqrt{2(\alpha-1)}\int_{\{x\in\Omega:
t<f(x)\le
t+1/h\}} |x|^{\alpha/2}dx\\
&&+\sqrt{2(\alpha-1)}\int_{\{x\in\Omega:  f(x)\ge t+1/h\}} |x
|^{\alpha/2}dx.
\end{eqnarray*}
Taking the limit $h\rightarrow\infty$ and noting that Theorem
\ref{propos4}, we have
\begin{equation}\label{eq2.12}P_{\mathcal{H}_{\alpha}}(E_t,\Omega)\le
\lim\sup_{h\rightarrow\infty}\| \nabla_{\mathcal{H}_{\alpha}}
v_{h}\|(\Omega)\le
\sqrt{2}m'(t)+\sqrt{2(\alpha-1)}\int^\infty_{-\infty}\big(\int_{\{x\in\Omega:
f(x)\ge t\}} |x |^{\alpha/2}dx\big)dt.
\end{equation}
Integrating (\ref{eq2.12}) and using (\ref{eqq2.9}) we obtain
\begin{equation*}\label{eqq2.9}\int^\infty_{-\infty}P_{\mathcal{H}_{\alpha}}(E_t,\Omega)dt\le
 \sqrt{2}\int_{\Omega}|\nabla  f|dx+\sqrt{2(\alpha-1)}\int_{\Omega}|f(x)||x|^{\alpha/2}dx\le \sqrt{2}\int_{\Omega}|\nabla_{\mathcal{H}_{\alpha}}  f|dx.\end{equation*}
Finally, by approximation and using the lower semi-continuity of the
$\alpha$-Hermite perimeter, we conclude that  (\ref{eq2.8}) holds
true for all $f\in \mathcal{BV}_{\mathcal{H}_{\alpha}}(\Omega)$ (see
Evans and Gariepy \cite{EG} for details).

\end{proof}

Finally, we develop some inequalities for   $\alpha$-HBV functions
and   $\alpha$-Hermite perimeters.

\begin{theorem}\label{thm2.7}
\item{{\rm (i)}} (Sobolev's inequality) For all $f\in
\mathcal{BV}_{\mathcal{H}_{\alpha}}( \mathbb{R}^d)$,
\begin{equation*}\|f\|_{L^{{d}/{(d-1)}}}\lesssim \| \nabla_{\mathcal{H}_{\alpha}} f\|( \mathbb{R}^d).
\end{equation*}

\item{{\rm (ii)}} (Isoperimetric inequality) Let $E$ be a bounded set
of finite $\alpha$-Hermite perimeter in $ \mathbb{R}^d$. Then
\begin{equation*}|E|^{1-{1}/{d}}\lesssim P_{\mathcal{H}_{\alpha}}(E).
\end{equation*}

\item{{\rm (iii)}} The above two statements are equivalent.
\end{theorem}

\begin{proof}(i) Choose
$$
f_k\in C^\infty_c(\mathbb{R}^d)\cap
\mathcal{BV}_{\mathcal{H}_{\alpha}}(\mathbb{R}^d), k=1,2,\ldots,
$$ such that
$$
\begin{cases}
f_k\rightarrow
f \quad\hbox{in}\quad L^1(\mathbb{R}^d);\\
\int_{\mathbb{R}^d} |\nabla_{\mathcal{H}_{\alpha}}
f_k(x)|dx\rightarrow
\parallel  \nabla_{\mathcal{H}_{\alpha}}  f\parallel(\mathbb{R}^d).
\end{cases}
$$

Then by Fatou's lemma  and the classical Gagliardo-Nirenberg-Sobolev
inequality (see \cite{EG}), we have
\begin{eqnarray*}\|f\|_{L^{{d}/{(d-1)}}}\le\liminf_{k\rightarrow\infty}\|f_k\|_{L^{{d}/{(d-1)}}}\lesssim
\lim_{k\rightarrow\infty}  \|\nabla f\|_{L^1}\lesssim
\lim_{k\rightarrow\infty}  \|\nabla_{\mathcal{H}_{\alpha}}
f\|_{L^1}=  \| \nabla_{\mathcal{H}_{\alpha}} f\|( \mathbb{R}^d),
\end{eqnarray*}where we have used the relation between the gradient
$\nabla$ and the $\alpha$-Hermite gradient $\nabla_{\mathcal{H}_{\alpha}}$ in
(\ref{eq3.11}).

(ii) We can show that (ii) is valid via letting $f=1_E$ in (i).

(iii) Obviously,  (i)$\Rightarrow$(ii) has been proved. In what
follows, we prove (ii)$\Rightarrow$(i). Assume that $0\le f\in
C^\infty_c(\mathbb{R}^d)$. By the coarea formula  in  Theorem
\ref{thm2.6} and (ii), we have
\begin{equation*}\label{equa2}
\int_{\mathbb{R}^d} |\nabla_{\mathcal{H}_{\alpha}}
f(x)|dx\approx\int_{0}^\infty
P_{\mathcal{H}_{\alpha}}(E_t)\,dt\gtrsim\int_{0}^\infty
|E_t|^{1-{1}/{d}}dt,
\end{equation*}
where $E_t=\big(\{x\in \mathbb{R}^d:\ f(x)>t\}\big)$. Let
$$f_t=\min\{t,f\}\ \ \&\ \ \chi(t)=\Big(\int_{\mathbb{R}^d}
(f_t(x))^{{d}/{(d-1)}}dx\Big)^{1-1/d}\ \ \forall\ t\in \mathbb{R}.
$$
It is easy to see that
$$\lim_{t\rightarrow \infty}\chi(t)=\Big(\int_{\mathbb{R}^d}
|f(x)|^{{d}/{(d-1)}}dx\Big)^{1-1/d}.$$ We can check that
$\chi(\cdot)$ is nondecreasing on $(0,\infty)$ and for $h>0$,
$$0\le \chi(t+h)-\chi(t)\le \Big(\int_{\mathbb{R}^d}
|f_{t+h}(x)-f_t(x)|^{{d}/{(d-1)}}dx\Big)^{1-1/d}\le
h|E_t|^{1-1/d}.$$ Then $\chi(\cdot)$
 is locally Lipschitz
and $\chi'(t)\le |E_t|^{1-1/d}$,  for a.e. $t\in (0,\infty)$.
Hence,
\begin{eqnarray*}
\Big(\int_{\mathbb{R}^d}
|f(x)|^{{d}/{(d-1)}}dx\Big)^{1-1/d}=\int^\infty_0 \chi'(t)dt\le
\int^\infty_0  |E_t|^{1-1/d}dt\lesssim\int_{\mathbb{R}^d}
|\nabla_{\mathcal{H}_{\alpha}} f(x)|dx.\end{eqnarray*}

\end{proof}

The following lemma gives some estimates for the  $\alpha$-Hermite
perimeter, which are different from the cases of the classical
perimeter.

\begin{lemma}\label{coro2.5} For any set $E$ in $\mathbb{R}^d$, denote by $sE$ the set $\{sx: x\in E\}$. The following statements
are valid:

\item{\rm (i)} If\, $0<s\le 1$, then
\begin{equation}\label{equa3.8}
s^{d+\alpha/2}P_{\mathcal{H}_{\alpha}}(E)\lesssim
P_{\mathcal{H}_{\alpha}}(sE)\lesssim
s^{d-1}P_{\mathcal{H}_{\alpha}}(E).\end{equation}

\item{\rm (ii)} If\, $s> 1$, then
\begin{equation}\label{equaa3.9}
s^{d-1}P_{\mathcal{H}_{\alpha}}(E)\lesssim
P_{\mathcal{H}_{\alpha}}(sE)\lesssim
s^{d+\alpha/2}P_{\mathcal{H}_{\alpha}}(E).
\end{equation}

\end{lemma}
\begin{proof}
By the definition of the $\alpha$-Hermite perimeter, we have
\begin{eqnarray*}
P_{\mathcal{H}_{\alpha}}(sE)&=& \sup_{\varphi\in \mathcal
{F}(\Omega)}\Big\{\int_{sE}
\mathrm{div}_{\mathcal{H}_{\alpha}}\varphi(x)dx\Big\}\\
&=&\sup_{\varphi\in \mathcal
{F}(\Omega)}\Big\{\int_{E}s^{d-1}\Big[\sum^d_{k=1}\frac{\partial}{\partial
x_k}\varphi_{d-k+1}(sx)+\sum^d_{k=1} \frac{\partial}{\partial
x_k}\varphi_{d+1+k}(sx)\Big]\\&&+s^{d+\alpha/2}\Big[\sum^d_{k=1}
\sqrt{\alpha-1}x_k|x|^{\frac{\alpha-2}{2}}\varphi_{d-k+1}(sx)-\sum^d_{k=1}
 \sqrt{\alpha-1}x_k|x|^{\frac{\alpha-2}{2}}\varphi_{d+1+k}(sx)\Big]dx\Big\}.\end{eqnarray*}

If  $0<s\le 1$, since \begin{eqnarray*} &&\sup_{\varphi\in \mathcal
{F}(\Omega)}\Big\{\int_{E}s^{d-1}\Big[\sum^d_{k=1}\frac{\partial}{\partial
x_k}\varphi_{d-k+1}(sx)+\sum^d_{k=1} \frac{\partial}{\partial
x_k}\varphi_{d+1+k}(sx)\Big]\\
&&\quad+s^{d+\alpha/2}\Big[\sum^d_{k=1}
\sqrt{\alpha-1}x_k|x|^{(\alpha-2)/2}\varphi_{d-k+1}(sx)-\sum^d_{k=1}
 \sqrt{\alpha-1}x_k|x|^{(\alpha-2)/2}\varphi_{d+1+k}(sx)\Big]dx\Big\}\\
&&\lesssim s^{d-1}\sup_{\varphi\in \mathcal
{F}(\Omega)}\Big\{\int_{E}\Big[\sum^d_{k=1}\frac{\partial}{\partial
x_k}\varphi_{d-k+1}(x)+\sum^d_{k=1} \frac{\partial}{\partial
x_k}\varphi_{d+1+k}(x)\Big]\\
&&\quad+ \Big[\sum^d_{k=1}
\sqrt{\alpha-1}x_k|x|^{(\alpha-2)/2}\varphi_{d-k+1}(x)-\sum^d_{k=1}
 \sqrt{\alpha-1}x_k|x|^{(\alpha-2)/2}\varphi_{d+1+k}(x)\Big]dx\Big\},
 \end{eqnarray*} then
$$P_{\mathcal{H}_{\alpha}}(sE)
 \lesssim  s^{d-1}P_{\mathcal{H}_{\alpha}}(E).$$
Moreover, since
\begin{eqnarray*}
&&\sup_{\varphi\in \mathcal
{F}(\Omega)}\Big\{\int_{E}s^{d-1}\Big[\sum^d_{k=1}\frac{\partial}{\partial
x_k}\varphi_{d-k+1}(sx)+\sum^d_{k=1} \frac{\partial}{\partial
x_k}\varphi_{d+1+k}(sx)\Big]\\
&&\quad+s^{d+\alpha/2}\Big[\sum^d_{k=1}
\sqrt{\alpha-1}x_k|x|^{(\alpha-2)/2}\varphi_{d-k+1}(sx)-\sum^d_{k=1}
 \sqrt{\alpha-1}x_k|x|^{(\alpha-2)/2}\varphi_{d+1+k}(sx)\Big]dx\Big\}\\
&&\gtrsim s^{d+\alpha/2}\sup_{\varphi\in \mathcal
{F}(\Omega)}\Big\{\int_{E}\Big[\sum^d_{k=1}\frac{\partial}{\partial
x_k}\varphi_{d-k+1}(x)+\sum^d_{k=1} \frac{\partial}{\partial
x_k}\varphi_{d+1+k}(x)\Big]\\
&&\quad+ \Big[\sum^d_{k=1}
\sqrt{\alpha-1}x_k|x|^{(\alpha-2)/2}\varphi_{d-k+1}(x)-\sum^d_{k=1}
 \sqrt{\alpha-1}x_k|x|^{(\alpha-2)/2}\varphi_{d+1+k}(x)\Big]dx\Big\},
 \end{eqnarray*} then
$$P_{\mathcal{H}_{\alpha}}(sE)
 \gtrsim  s^{d+\alpha/2}P_{\mathcal{H}_{\alpha}}(E).$$
Therefore, (\ref{equa3.8}) is proved. The inequalities in (\ref{equaa3.9})
can be  proved in a similar way.

\end{proof}

An immediate corollary of the above lemma is as follows.

\begin{corollary}\label{coro2.5}
Let $B(0,s)$ be the open ball centered at $0$ with radius $s$, where
$0$ is the origin of\, $\mathbb{R}^d$.

$\mathrm{(i)}$ If\, $0<s\le 1$,
\begin{equation*}\label{equa3.8}s^{d+\alpha/2}\lesssim
P_{\mathcal{H}_{\alpha}}(B(0,s))\lesssim s^{d-1}.\end{equation*}

$\mathrm{(ii)}$ If\, $s> 1$,
\begin{equation*}\label{equa3.9}s^{d-1} \lesssim P_{\mathcal{H}_{\alpha}}(B(0,s))\lesssim s^{d+\alpha/2}.\end{equation*}

\end{corollary}

\begin{remark}\label{rem2.13} It should be noted that    the set $E$ and its
complementary set have the same perimeter, while this fact plays an
important role during the proof of the main theorem in
\cite{Barozzi}. But unfortunately, for the case of the
$\alpha$-Hermite perimeter, the above fact doesn't hold. For
example, let $E=B(0,r)$ with $r>0$. By the definition of the
$\alpha$-Hermite perimeter, Corollary \ref{coro2.5} indicates that
\begin{equation}\label{eq-converse-example}
P_{\mathcal{H}_{\alpha}}(B(0,r)^c)\gtrsim\int_{B(0,r)^c}|x|^{\alpha/2}dx=\infty>P_{\mathcal{H}_{\alpha}}(B(0,r)).
\end{equation}
\end{remark}
Next we introduce the    so called  restricted $\alpha$-Hermite
perimeter as follows.

\begin{definition}\label{def-5.1}
The restricted $\alpha$-Hermite perimeter of  $E\subseteq \mathbb
R^{d}$ can be defined as follows:
$$\widetilde{P}_{\mathcal{H}_{\alpha}}(E)=
\sup_{\varphi\in \mathcal {F}_R(\mathbb R^{d})}\Big\{\int_E
\mathrm{div}_{\mathcal{H}_{\alpha}}\varphi(x)dx\Big\},$$
   where
$\mathcal{F}_R(\mathbb R^{d})$ denotes the class of all functions
$\varphi=(\varphi_1,\varphi_2,\ldots,\varphi_{2d})\in C^1_c(\mathbb
R^{d};\mathbb R^{2d})$  such that $$\parallel
\varphi\parallel_{\infty}=\sup_{x\in \mathbb R^{d}}(\mid
\varphi_1(x)\mid^2+\cdots+\mid \varphi_{2d}(x)\mid^2)^{1/2}\le 1 $$
and \begin{equation}\label{equa5.1} \int_{\mathbb
R^{d}}\Big[\sum^d_{k=1} \sqrt{\alpha-1}x_k|x|^{
{(\alpha-2)}/{2}}\varphi_{d-k+1}(x)-\sum^d_{k=1}
 \sqrt{\alpha-1}x_k|x|^{{(\alpha-2)}/{2}}\varphi_{d+1+k}(x)\Big]dx=0.
 \end{equation}
 \end{definition}
It is obvious that for any set $E$ in $\mathbb{R}^d$,
$$\widetilde{P}_{\mathcal{H}_{\alpha}}(
E)\le  {P}_{\mathcal{H}_{\alpha}}( E). $$

\begin{lemma}\label{lem-5.1}
 For any set $E$ in $\mathbb{R}^d$ with finite restricted $\alpha$-Hermite perimeter,
 $$\widetilde{P}_{\mathcal{H}_{\alpha}}( E)=\widetilde{P}_{\mathcal{H}_{\alpha}}( E^c).$$
\end{lemma}
\begin{proof} For any $\varphi\in \mathcal {F}_R(\mathbb
R^{d})$, via the classical divergence theorem and noting the compact
support of $\varphi$,  we have
\begin{eqnarray*}\int_E
\mathrm{div}_{\mathcal{H}_{\alpha}}\varphi(x)dx&=&\int_E
\mathrm{div} (\varphi_1(x),\ldots,\varphi_d(x))dx+\int_E
\mathrm{div}
(\varphi_{d+1}(x),\ldots,\varphi_{2d}(x))dx\\
&&+\int_{E} \Big[\sum^d_{k=1}
\sqrt{\alpha-1}x_k|x|^{{(\alpha-2)}/{2}}\varphi_{d-k+1}(x)-\sum^d_{k=1}
 \sqrt{\alpha-1}x_k|x|^{{(\alpha-2)}/{2}}\varphi_{d+1+k}(x)\Big]dx\\
 &=&-\int_{\partial E^c}(\varphi_1(x),\ldots,\varphi_d(x))\cdot \vec{n} ds-\int_{\partial E^c}(\varphi_{d+1}(x),\ldots,\varphi_{2d}(x))\cdot \vec{n} ds\\
 &&-\int_{E^c} \Big[\sum^d_{k=1}
\sqrt{\alpha-1}x_k|x|^{{(\alpha-2)}/{2}}\varphi_{d-k+1}(x)-\sum^d_{k=1}
 \sqrt{\alpha-1}x_k|x|^{{(\alpha-2)}/{2}}\varphi_{d+1+k}(x)\Big]dx\\
 &&+\int_{\mathbb R^{d}} \Big[\sum^d_{k=1}
\sqrt{\alpha-1}x_k|x|^{{(\alpha-2)}/{2}}\varphi_{d-k+1}(x)-\sum^d_{k=1}
 \sqrt{\alpha-1}x_k|x|^{{(\alpha-2)}/{2}}\varphi_{d+1+k}(x)\Big]dx\\
 &=& -\int_{E^c}
\mathrm{div}_{\mathcal{H}_{\alpha}}\varphi(x)dx,\end{eqnarray*}
where we have used (\ref{equa5.1}) in the last step.

Due to the arbitrariness of $\varphi$, taking the supremum implies

 $$\widetilde{P}_{\mathcal{H}_{\alpha}}( E)=\widetilde{P}_{\mathcal{H}_{\alpha}}(
 E^c).
 $$
\end{proof}

Using similar methods, we conclude that
$\widetilde{P}_{\mathcal{H}_{\alpha}}( \cdot)$ enjoys several same
properties as ${P}_{\mathcal{H}_{\alpha}}( \cdot)$. In the sequel,
$P_{\mathcal{H}_{\alpha}}( \cdot)$  will be used in Section
\ref{sec-3}, while $\widetilde{P}_{\mathcal{H}_{\alpha}}( \cdot)$
will be used to investigate the mean curvature of a set with finite
restricted $\alpha$-Hermite perimeter. For convenience, we give
several properties for $\widetilde{P}_{\mathcal{H}_{\alpha}}(
\cdot)$
 and omit the details of the   proof.
\begin{lemma}\label{le-1.17}
The restricted $\alpha$-Hermite perimeter is  lower semi-continuous. Precisely, if $ 1_{E_k}\rightarrow 1_E \ \ \text{in}\ \
L^1_{\mathrm{loc}}(\Omega)$, where $E_k$ and $E$ are subsets of $\Omega$ for $k=1,2,\ldots$, then
\begin{equation*}
\lim\inf_{k\rightarrow\infty}\widetilde{P}_{\mathcal{H}_{\alpha}}(E_k,\Omega)\ge \widetilde{P}_{\mathcal{H}_{\alpha}}(E,\Omega).
\end{equation*}
\end{lemma}

Similar to Lemma \ref{lem3.6},  we have
\begin{lemma}\label{le-1.18}
For any  subsets $E$ in $
\Omega$, we have
$$\widetilde{P}_{\mathcal{H}_{\alpha}}(E\cap F, \Omega)+\widetilde{P}_{\mathcal{H}_{\alpha}}(E\cup F, \Omega)\le \widetilde{P}_{\mathcal{H}_{\alpha}}(E, \Omega)+\widetilde{P}_{\mathcal{H}_{\alpha}}( F, \Omega).$$
\end{lemma}

In the same manner, we can list the　 analogues of  previous results
for $\widetilde{P}_{\mathcal{H}_{\alpha}}(\cdot)$, such as the
coarea formula, the Sobolev inequality, the isoperimetric
inequality.

\section{$\alpha$-HBV capacity}\label{sec-3}

Based on the results on $\alpha$-Hermite BV spaces,  we introduce
the $\alpha$-HBV capacity and investigate its properties.
\begin{definition}\label{defi2}%Let $\Omega\subseteq  \mathbb R^d$ be an open set.
For a set  $E\subseteq  \mathbb R^d$, let $\mathcal{A}(E,
\mathcal{BV}_{\mathcal{H}_{\alpha}}(\mathbb R^d))$ be the class of
admissible functions on $\mathbb R^d$, i.e., functions $f\in
\mathcal{BV}_{\mathcal{H}_{\alpha}}(\mathbb R^d)$ satisfying $0\le
f\le 1$ and $f=1$ in a neighborhood of $E$ (an open set containing
$E$). The $\alpha$-HBV capacity of $E$ is defined by
\begin{equation*}\label{equa6}
\mathrm{cap}(E, \mathcal{BV}_{\mathcal{H}_{\alpha}}(\mathbb
R^d)):=\inf_{f\in\mathcal{A}(E, \mathcal{BV}_{\mathcal{H}_{\alpha}}(\mathbb
R^d))}\Big\{\parallel f
\parallel_{L^1}+\parallel \nabla_{\mathcal{H}_{\alpha}} f \parallel(\mathbb R^d)\Big\}.
\end{equation*}
\end{definition}

Via the co-area formula for $\alpha$-HBV functions in Theorem
\ref{thm2.6},  we obtain the following basic assertions.
\begin{theorem}
\label{lemm4} A geometric description of the $\alpha$-HBV capacity
of a set in  $\mathbb R^d$ is given as follows:
\item{{\rm (i)}}  For any set $K\subseteq \mathbb R^d$,
\begin{equation*}
\mathrm{cap}(K, \mathcal{BV}_{\mathcal{H}_{\alpha}}(\mathbb R^d))\approx\inf_A \Big\{|A|+P_{\mathcal{H}_{\alpha}}(A)\Big\},
\end{equation*}where the infimum is taken over all   sets $A\subseteq \mathbb R^d$ such that $K\subseteq
int(A)$.

\item{{\rm (ii)}} For any compact  set $K\subseteq \mathbb R^d$,
\begin{equation*}
\mathrm{cap}(K, \mathcal{BV}_{\mathcal{H}_{\alpha}}(\mathbb
R^d))\approx\inf_A \Big\{|A|+P_{\mathcal{H}_{\alpha}}(A)\Big\},
\end{equation*}
where the infimum is taken over all bounded open sets $A$ with smooth boundary in $\mathbb R^d$ containing $K$.
\end{theorem}
\begin{proof}(i)
If $A\subseteq \mathbb R^d$ with $K\subseteq  int(A)$ and
$|A|+P_{\mathcal{H}_{\alpha}}(A)<\infty$, $1_A\in \mathcal{A}(K,
\mathcal{BV}_{\mathcal{H}_{\alpha}}(\mathbb R^d))$ and hence,
$$\mathrm{cap}(K, \mathcal{BV}_{\mathcal{H}_{\alpha}}(\mathbb R^d))\le |A|+ P_{\mathcal{H}_{\alpha}}(A).$$
By taking the infimum over all such sets $A$, we obtain
$$\mathrm{cap}(K, \mathcal{BV}_{\mathcal{H}_{\alpha}}(\mathbb R^d))\le \inf_A \Big\{|A|+ P_{\mathcal{H}_{\alpha}}(A)\Big\}.$$

In order to prove the reverse inequality, we may assume that
$\mathrm{cap}(K, \mathcal{BV}_{\mathcal{H}_{\alpha}}(\mathbb
R^d))<\infty.$ Let $\varepsilon>0$ and $f\in \mathcal{A}(K,
\mathcal{BV}_{\mathcal{H}_{\alpha}}(\mathbb R^d)) $ such that
$$\parallel f
\parallel_{L^1}+\parallel \nabla_{\mathcal{H}_{\alpha}} f
\parallel(\mathbb R^d)<\mathrm{cap}(K,
\mathcal{BV}_{\mathcal{H}_{\alpha}}(\mathbb R^d))+\varepsilon.$$
Using the co-area formula (\ref{eq2.8}) and the Cavalieri principle,
we have
\begin{eqnarray*}
\int_{\mathbb R^d}f(x)dx+\parallel \nabla_{\mathcal{H}_{\alpha}} f \parallel(\mathbb R^d)&\approx&\int^1_0\Big[\Big|\Big\{x\in \mathbb R^d: f(x)>t\Big\}\Big|+ P_{\mathcal{H}_{\alpha}}(\{x\in \mathbb R^d: f(x)>t\}\Big]dt\\
&\gtrsim&   \inf_A \{|A|+ P_{\mathcal{H}_{\alpha}}(A)\},
\end{eqnarray*}
where we have used the fact:
$K\subseteq int\Big\{x\in \mathbb R^d: f(x)>t\Big\}$ for
$0<t<1.$
Then
$$ \inf_A \Big\{|A|+ P_{\mathcal{H}_{\alpha}}(A)\Big\}\lesssim\mathrm{cap}(K, \mathcal{BV}_{\mathcal{H}_{\alpha}}(\mathbb R^d))+\varepsilon.$$
The desired inequality now follows by letting
$\varepsilon\rightarrow 0$.

(ii) Using the co-area formula (\ref{eq2.8})  and the Cavalieri
principle again,    we can also  conclude that  (ii) is valid
similar to the proof  of (i) and so we omit the details here.
\end{proof}

\subsection{Measure-theoretic nature of $\alpha$-HBV capacity}\label{sec-3.1}

\begin{theorem}\label{thm8} Assume $A,B$ are  subsets of $ \mathbb R^d$.

$\mathrm{(i)}$ $$\mathrm{cap}(\emptyset,
\mathcal{BV}_{\mathcal{H}_{\alpha}}(\mathbb R^d))=0.$$

$\mathrm{(ii)}$ If  $A\subseteq B$, then $$\mathrm{cap}(A,
\mathcal{BV}_{\mathcal{H}_{\alpha}}(\mathbb R^d))\le \mathrm{cap}(B,
\mathcal{BV}_{\mathcal{H}_{\alpha}}(\mathbb   R^d)).$$

%$\mathrm{(iii)}$ $$\mathrm{cap}(\delta_\lambda A,
%\mathcal{BV}_{\mathcal{H}_{\alpha}}(\mathbb R^d))=\lambda^{d-1}\mathrm{cap}(A,
%\mathcal{BV}_{\mathcal{H}_{\alpha}}(\mathbb R^d)),$$ where $\delta_\lambda
%A=\{\delta_\lambda g: g\in A\}$.
%
%
%$\mathrm{(iv)}$ $$\mathrm{cap}(L_{y'} A, \mathcal{BV}_{\mathcal{H}_{\alpha}}(\mathbb
%G^2_\alpha))=\mathrm{cap}(A, \mathcal{BV}_{\mathcal{H}_{\alpha}}(\mathbb R^d))$$ for any
%vertical translation $L_{y'}$ with $y'\in\mathbb{R}$.
%
%$\mathrm{(v)}$ $$\mathrm{cap}(B ((0,y),r), \mathcal{BV}_{\mathcal{H}_{\alpha}}(\mathbb
%G^2_\alpha))=r^{d-1}\mathrm{cap}(B ((0,0),1), \mathcal{BV}_{\mathcal{H}_{\alpha}}(\mathbb
%G^2_\alpha)) $$ for any $y\in\mathbb{R}$.

$\mathrm{(iii)}$ \begin{eqnarray*}\mathrm{cap}(A\cup B,
\mathcal{BV}_{\mathcal{H}_{\alpha}}(\mathbb R^d))&+&\mathrm{cap}(A
\cap B, \mathcal{BV}_{\mathcal{H}_{\alpha}}(\mathbb R^d))\\ &\le&
\mathrm{cap}(A, \mathcal{BV}_{\mathcal{H}_{\alpha}}(\mathbb
R^d))+\mathrm{cap}(B, \mathcal{BV}_{\mathcal{H}_{\alpha}}(\mathbb
R^d)).\end{eqnarray*}

$\mathrm{(iv)}$ If $A_k, k=1,2,\ldots$, are   subsets in $\mathbb
R^d$, then
$$\mathrm{cap}(\cup^\infty_{k=1}A_k, \mathcal{BV}_{\mathcal{H}_{\alpha}}(\mathbb R^d))\le \sum^\infty_{k=1}\mathrm{cap}(A_k, \mathcal{BV}_{\mathcal{H}_{\alpha}}(\mathbb R^d)).$$

$\mathrm{(v)}$ For any sequence $\{A_k\}^\infty_{k=1}$ of
 subsets of $\mathbb R^d$ with $A_1\subseteq A_2\subseteq A_3\subseteq\cdots,$ $$\lim_{k\rightarrow\infty}\mathrm{cap}(A_k,
\mathcal{BV}_{\mathcal{H}_{\alpha}}(\mathbb
R^d))=\mathrm{cap}(\cup^\infty_{k=1}A_k,
\mathcal{BV}_{\mathcal{H}_{\alpha}}(\mathbb R^d)).$$

$\mathrm{(vi)}$ If $A_k, k=1,2,\ldots$, are compact sets in $\mathbb
R^d$ and   $A_1\supseteq A_2\supseteq A_3\supseteq\cdots$,
 then$$\lim_{k\rightarrow
\infty}\mathrm{cap}(A_k)=\mathrm{cap}(\cap^\infty_{k=1}A_k).$$
\end{theorem}
 \begin{proof}

(i)-(ii). Statements (i) and (ii) are the evident consequences of
Definition \ref{defi2}.

(iii). Without loss of generality, we may assume
$$\mathrm{cap}(A, \mathcal{BV}_{\mathcal{H}_{\alpha}}(\mathbb R^d))+\mathrm{cap}(B,
\mathcal{BV}_{\mathcal{H}_{\alpha}}(\mathbb R^d))<\infty.$$ For any
$\varepsilon>0$, there are two functions $\phi\in
\mathcal{A}(A,\mathcal{BV}_{\mathcal{H}_{\alpha}}(\mathbb R^d))$ and
$\psi\in \mathcal{A}(B,\mathcal{BV}_{\mathcal{H}_{\alpha}}(\mathbb
R^d))$, such that
$$\begin{cases}
\parallel  \phi\parallel_{L^1}+\parallel \nabla_{\mathcal{H}_{\alpha}} \phi\parallel(\mathbb R^d)<
\mathrm{cap}(A, \mathcal{BV}_{\mathcal{H}_{\alpha}}(\mathbb R^d))+\frac{\varepsilon}{2};\\
\parallel  \psi\parallel_{L^1}+\parallel \nabla_{\mathcal{H}_{\alpha}} \psi\parallel(\mathbb R^d)< \mathrm{cap}(B,
\mathcal{BV}_{\mathcal{H}_{\alpha}}(\mathbb
R^d))+\frac{\varepsilon}{2}.
\end{cases}$$
Let
$$\varphi_1=\max\{\phi,\psi\}\ \ \&\ \ {\varphi}_2=\min\{\phi,\psi\}.$$
It is easy to see that
$$
\varphi_1\in \mathcal{A}(A\cup B,
\mathcal{BV}_{\mathcal{H}_{\alpha}}(\mathbb R^d))\ \ \& \ \
\varphi_2\in \mathcal{A}(A\cap B,
\mathcal{BV}_{\mathcal{H}_{\alpha}}(\mathbb R^d)).
$$
Then by Theorem \ref{theorem2},
\begin{eqnarray*}
&&\mathrm{cap}(A\cup B, \mathcal{BV}_{\mathcal{H}_{\alpha}}(\mathbb
R^d))+\mathrm{cap}(A
\cap B, \mathcal{BV}_{\mathcal{H}_{\alpha}}(\mathbb R^d))\\
&&\le \int_{\mathbb
R^d}\varphi_1(x)dx+\int_{\mathbb R^d}\varphi_2(x)dx+
\parallel \nabla_{\mathcal{H}_{\alpha}} \varphi_1\parallel(\mathbb R^d)+ \parallel \nabla_{\mathcal{H}_{\alpha}} \varphi_2\parallel(\mathbb R^d)\\
&&\le \int_{\mathbb
R^d}\phi(x)dx+\int_{\mathbb R^d}\psi(x)dx \parallel \nabla_{\mathcal{H}_{\alpha}} \phi\parallel(\mathbb R^d)+ \parallel \nabla_{\mathcal{H}_{\alpha}} \psi\parallel(\mathbb R^d)\\
&&\le \mathrm{cap}(A, \mathcal{BV}_{\mathcal{H}_{\alpha}}(\mathbb
R^d))+\mathrm{cap}(B, \mathcal{BV}_{\mathcal{H}_{\alpha}}(\mathbb
R^d))+\varepsilon.
\end{eqnarray*}
The assertion (iii) is proved.

(iv).  Suppose
$$
\sum^\infty_{k=1}\mathrm{cap}(A_k,
\mathcal{BV}_{\mathcal{H}_{\alpha}}(\mathbb R^d))<\infty.
$$
For any $\varepsilon>0$ and $k=1,2,\ldots$, there is $ f_k\in
\mathcal{A}(A_k, \mathcal{BV}_{\mathcal{H}_{\alpha}}(\mathbb R^d)) $
such that
$$\parallel  f_k\parallel_{L^1}+\parallel \nabla_{\mathcal{H}_{\alpha}} f_k\parallel(\mathbb R^d)<\mathrm{cap}(A_k, \mathcal{BV}_{\mathcal{H}_{\alpha}}(\mathbb R^d))+\frac{\varepsilon}{2^k}.$$
Setting $f=\sup_kf_k$ gives
$$\int_{\mathbb R^d}f(x)dx\le \sum^\infty_{k=1}\int_{\mathbb R^d}f_k(x)dx<\sum^\infty_{k=1}\mathrm{cap}(A_k, \mathcal{BV}_{\mathcal{H}_{\alpha}}(\mathbb R^d))+\frac{\varepsilon}{2^k}<\infty,$$
which implies $f\in L^1(\mathbb R^d)$.

Via the lower semicontinuity (\ref{equation1}) of the
$\alpha$-Hermite variation we get
\begin{eqnarray*}\int_{\mathbb R^d}f(x)dx+\parallel \nabla_{\mathcal{H}_{\alpha}} f\parallel(\mathbb R^d)&\le&\sum^\infty_{k=1}\int_{\mathbb R^d}f_k(x)dx+\lim\inf_{k\rightarrow
\infty}\parallel \nabla_{\mathcal{H}_{\alpha}} \max\{f_1,\cdots,f_k\}\parallel(\mathbb R^d)\\
&\le& \sum^\infty_{k=1}\int_{\mathbb R^d}f_k(x)dx+\sum^\infty_{k=1}\parallel \nabla_{\mathcal{H}_{\alpha}} f_k\parallel(\mathbb R^d)\\
&\le& \sum^\infty_{k=1}\mathrm{cap}(A_k,
\mathcal{BV}_{\mathcal{H}_{\alpha}}(\mathbb R^d))+\varepsilon.
\end{eqnarray*}
Then we have $f\in \mathcal{A}(\cup^\infty_{k=1}A_k,
\mathcal{BV}_{\mathcal{H}_{\alpha}}(\mathbb R^d))$ and this
completes the proof of (iv) via letting $\varepsilon\to 0$.

(v). It is obvious that $$\lim_{k\rightarrow\infty}\mathrm{cap}(A_k,
\mathcal{BV}_{\mathcal{H}_{\alpha}}(\mathbb R^d))\le
\mathrm{cap}(\cup^\infty_{k=1}A_k,
\mathcal{BV}_{\mathcal{H}_{\alpha}}(\mathbb R^d)).$$ The equality
holds if
$$\lim_{k\rightarrow\infty}\mathrm{cap}(A_k, \mathcal{BV}_{\mathcal{H}_{\alpha}}(\mathbb R^d))=\infty.$$
Let $\varepsilon > 0$ and assume
$$\lim_{k\rightarrow\infty}\mathrm{cap}(A_k, \mathcal{BV}_{\mathcal{H}_{\alpha}}(\mathbb R^d))<\infty.$$
For   $k = 1, 2, \ldots$, there is
$$
f_k\in \mathcal{A}(A_k, \mathcal{BV}_{\mathcal{H}_{\alpha}}(\mathbb
R^d))
$$
such that
$$\parallel  f_k\parallel_{L^1}+\parallel \nabla_{\mathcal{H}_{\alpha}} f_k\parallel(\mathbb R^d)<\mathrm{cap}(A_k, \mathcal{BV}_{\mathcal{H}_{\alpha}}(\mathbb R^d))+\frac{\varepsilon}{2^k}.$$
Set
$$
\begin{cases}
\phi_k=\max_{1\le i\le k}f_i=\max\{\phi_{k-1},\ f_{k}\};\\
\phi_0=0;\\
A_0=\emptyset;\\
\varphi_k=\min\{\phi_{k-1},\ f_k\}.
\end{cases}
$$
Then
$$\phi_k,\varphi_k\in \mathcal{BV}_{\mathcal{H}_{\alpha}}(\mathbb R^d)\ \ \&\ \ A_{k-1}\subseteq int\{x\in \mathbb R^d: \varphi_k(x)=1\}.
$$
Since $\phi_k=\max\{\phi_{k-1}, \phi_k\}$, an application of Theorem
\ref{theorem2} derives
$$
\parallel \nabla_{\mathcal{H}_{\alpha}} \max\{\phi_{k-1}, \phi_k\}\parallel(\mathbb R^d)+\parallel \nabla_{\mathcal{H}_{\alpha}} \min\{\phi_{k-1}, \phi_k\}\parallel(\mathbb R^d)\le
 \parallel \nabla_{\mathcal{H}_{\alpha}} \phi_{k-1}\parallel(\mathbb R^d)
+\parallel \nabla_{\mathcal{H}_{\alpha}} \phi_{k}\parallel(\mathbb
R^d),
$$
and then
\begin{eqnarray*}
&&\parallel   \phi_{k}\parallel_{L^1}+\parallel
\nabla_{\mathcal{H}_{\alpha}} \phi_{k}\parallel(\mathbb
R^d)+\mathrm{cap}(A_{k-1},
\mathcal{BV}_{\mathcal{H}_{\alpha}}(\mathbb R^d))\\
&&\le \parallel   \phi_{k}\parallel_{L^1}+ \parallel
\nabla_{\mathcal{H}_{\alpha}} \phi_{k}\parallel(\mathbb R^d)
+\parallel   \varphi_{k}\parallel_{L^1}+\parallel \nabla_{\mathcal{H}_{\alpha}} \varphi_{k}\parallel(\mathbb R^d)\\
&&\le  \parallel   \phi_{k}\parallel_{L^1}+\parallel
\phi_{k-1}\parallel_{L^1}+\parallel \nabla_{\mathcal{H}_{\alpha}}
\phi_{k}\parallel(\mathbb R^d)+\parallel \nabla_{\mathcal{H}_{\alpha}} \phi_{k-1}\parallel(\mathbb R^d)\\
&&\le \parallel   \phi_{k-1}\parallel_{L^1}+\parallel
\nabla_{\mathcal{H}_{\alpha}} \phi_{k-1}\parallel(\mathbb
R^d)+\mathrm{cap}(A_k, \mathcal{BV}_{\mathcal{H}_{\alpha}}(\mathbb
R^d))+\frac{\varepsilon}{2^k},
\end{eqnarray*}
where we have used the fact that $A_{k-1}\subseteq A_{k}$.
Therefore,
\begin{eqnarray*}
&&\parallel   \phi_{k}\parallel_{L^1}+\parallel
\nabla_{\mathcal{H}_{\alpha}} \phi_{k}\parallel(\mathbb
R^d)-\parallel \phi_{k-1}\parallel_{L^1}-
\parallel \nabla_{\mathcal{H}_{\alpha}} \phi_{k-1}\parallel(\mathbb R^d)\\
&&\le \mathrm{cap}(A_{k},
\mathcal{BV}_{\mathcal{H}_{\alpha}}(\mathbb
R^d))-\mathrm{cap}(A_{k-1},
\mathcal{BV}_{\mathcal{H}_{\alpha}}(\mathbb
R^d))+\frac{\varepsilon}{2^k}.
\end{eqnarray*}

By adding the above inequalities from $k=1$ to $k=j$,  we get
$$\parallel   \phi_{j}\parallel_{L^1}+\parallel \nabla_{\mathcal{H}_{\alpha}} \phi_j\parallel(\mathbb R^d)\le \mathrm{cap}(A_{j},
\mathcal{BV}_{\mathcal{H}_{\alpha}}(\mathbb R^d))+\varepsilon.$$ Let
$\tilde{\phi}=\lim_{j\rightarrow\infty}\phi_j$. Via the monotone
convergence theorem,  we obtain
$$\int_{\mathbb R^d}\tilde{\phi}(x)dx=\lim_{j\rightarrow\infty}\int_{\mathbb R^d}\phi_j(x)dx\le\lim_{j\rightarrow\infty}\mathrm{cap}(A_{j},
\mathcal{BV}_{\mathcal{H}_{\alpha}}(\mathbb R^d))+\varepsilon.$$
Then via the  lower semicontinuity (\ref{equation1}) of the
$\alpha$-Hermite variation, we have
$$ \tilde{\phi}\in \mathcal{A}(\cup^\infty_{j=1}A_j,
\mathcal{BV}_{\mathcal{H}_{\alpha}}(\mathbb R^d))
$$
and
\begin{eqnarray*} \mathrm{cap}(\cup^\infty_{j=1}A_j,
\mathcal{BV}_{\mathcal{H}_{\alpha}}(\mathbb R^d))&\le& \parallel   \tilde{\phi} \parallel_{L^1}+\parallel \nabla_{\mathcal{H}_{\alpha}} \tilde{\phi}\parallel(\mathbb R^d)\\
&\le& \liminf_{j\rightarrow\infty}\big(\int_{\mathbb R^d}\phi_j(x)dx+\parallel \nabla_{\mathcal{H}_{\alpha}} \phi_j\parallel(\mathbb R^d)\big)\\
&\le& \lim_{j\rightarrow\infty}\mathrm{cap}(A_{j},
\mathcal{BV}_{\mathcal{H}_{\alpha}}(\mathbb R^d))+\varepsilon.
\end{eqnarray*}

(vi). Let $A=\cap^\infty_{k=1}A_k$.  By monotonicity,
$$\mathrm{cap}({\bigcap^\infty_{k=1}A_k}, \mathcal{BV}_{\mathcal{H}_{\alpha}}(\mathbb R^d))\le \lim_{k\rightarrow \infty}
\mathrm{cap}(A_k, \mathcal{BV}_{\mathcal{H}_{\alpha}}(\mathbb
R^d)).$$ Let $U$ be an open set containing $A$. Then by the
compactness of $A$,  we know that $A_k\subseteq U$ for all
sufficiently large $k$. Therefore,
$$\lim_{k\rightarrow \infty}
\mathrm{cap}(A_k, \mathcal{BV}_{\mathcal{H}_{\alpha}}(\mathbb
R^d))\le \mathrm{cap}(U, \mathcal{BV}_{\mathcal{H}_{\alpha}}(\mathbb
R^d)).$$ Corollary \ref{cor8} implies that
 $\mathrm{cap}(\cdot, \mathcal{BV}_{\mathcal{H}_{\alpha}}(\mathbb R^d))$    is an outer
capacity. Then we obtain the claim by taking infimum over all open
sets $U$ containing $A$.

\end{proof}

\begin{corollary}\label{cor8}

\item{{\rm (i)}} If $E\subseteq \mathbb R^d$, then
$$\mathrm{cap}(E,
\mathcal{BV}_{\mathcal{H}_{\alpha}}(\mathbb
R^d))=\inf_{\mathrm{open}\,
O\supseteq E}\Big\{\mathrm{cap}(O,
\mathcal{BV}_{\mathcal{H}_{\alpha}}(\mathbb R^d)) \Big\}.$$

 \item{{\rm (ii)}} If $E\subseteq \mathbb R^d$ is a Borel set, then
$$\mathrm{cap}(E,
\mathcal{BV}_{\mathcal{H}_{\alpha}}(\mathbb
R^d))=\sup_{\mathrm{compact}\, K\subseteq E}\Big\{\mathrm{cap}(K,
\mathcal{BV}_{\mathcal{H}_{\alpha}}(\mathbb R^d))\Big\}.$$
\end{corollary}

\begin{proof} (i). The statement (ii) of Theorem \ref{thm8} implies
$$\mathrm{cap}(E,
\mathcal{BV}_{\mathcal{H}_{\alpha}}(\mathbb
R^d))\le\inf_{\mathrm{open}\,
O\supseteq E}\Big\{\mathrm{cap}(O,
\mathcal{BV}_{\mathcal{H}_{\alpha}}(\mathbb R^d))\Big\}.$$

To prove the reverse inequality, we may assume
$$
\mathrm{cap}(E, \mathcal{BV}_{\mathcal{H}_{\alpha}}(\mathbb
R^d))<\infty.
$$
Via Definition \ref{defi2}, for any $\varepsilon>0$,  there is $f\in
\mathcal{A}(E, \mathcal{BV}_{\mathcal{H}_{\alpha}}(\mathbb R^d))$
such that
$$ \parallel   f \parallel_{L^1}+\parallel \nabla_{\mathcal{H}_{\alpha}} f\parallel(\mathbb R^d)< \mathrm{cap}(E,
\mathcal{BV}_{\mathcal{H}_{\alpha}}(\mathbb R^d))+\varepsilon.$$
Hence, there exists an open set $O\supseteq E$ such that $f=1$ on
$O$, which implies
$$\mathrm{cap}(O,
\mathcal{BV}_{\mathcal{H}_{\alpha}}(\mathbb R^d))\le  \parallel   f
\parallel_{L^1}+\parallel \nabla_{\mathcal{H}_{\alpha}} f\parallel(\mathbb R^d)<
\mathrm{cap}(E, \mathcal{BV}_{\mathcal{H}_{\alpha}}(\mathbb
R^d))+\varepsilon.$$ Therefore,
$$\mathrm{cap}(E,
\mathcal{BV}_{\mathcal{H}_{\alpha}}(\mathbb R^d))\ge
\inf_{\mathrm{open}\, O\supseteq E}\Big\{\mathrm{cap}(O, \mathcal{BV}_{\mathcal{H}_{\alpha}}(\mathbb
R^d))\Big\}.$$

(ii). This follows from  (v) and (vi) of Theorem \ref{thm8}.
\end{proof}

In \cite{huang1}, the authors   introduced the $\alpha$-Hermite
Sobolev $p$-capacity associated with the Hermite operator
$\mathcal{H}_\alpha$ and investigated the related topics.  Following
from \cite{huang1}. we give the definition of the Sobolev
$1$-capacity.
\begin{definition}\label{def-1.1}
  Let  $E\subset \mathbb R^d$ and
  $$\mathcal{A}_1(E)=\Big\{f\in W^{1,1}_{\mathcal{H}_{\alpha}}(\mathbb R^d):\ E\subset\{x\in\mathbb R^d: f(x)\geq 1\}^\circ\Big\}.$$
  The Sobolev $1$-capacity of $E$ is defined by
  $$ Cap^{\mathcal{H}_{\alpha}}_1(E)=\inf_{f\in\mathcal{A}_1(E)}\Big\{\|f\|_{W^{1,1}_{\mathcal{H}_{\alpha}}}\Big\}.$$
 \end{definition}

\begin{proposition}For any set $E\subseteq \mathbb{R}^d$, then
$$\mathrm{cap}(E,
\mathcal{BV}_{\mathcal{H}_{\alpha}}(\mathbb R^d))\lesssim
Cap^{\mathcal{H}_{\alpha}}_1(E).$$
\end{proposition}
\begin{proof}For any $f\in\mathcal{A}_1(E)$, via (i) of Lemma
\ref{lem1.1}, we have
\begin{eqnarray*}\|f\|_{W^{1,1}_{\mathcal{H}_{\alpha}}}&=&\int_{\mathbb{R}^d}
|\nabla_{\mathcal{H}_{\alpha}}f(x)|dx+\int_{\mathbb{R}^d} | f(x)|dx\\
&\gtrsim&  \int^1_0|\{x\in \mathbb R^d: f(x)>t\}|+
P_{\mathcal{H}_{\alpha}}(\{x\in \mathbb
R^d: f(x)>t\})dt\\
&\gtrsim& \mathrm{cap}(E,
\mathcal{BV}_{\mathcal{H}_{\alpha}}(\mathbb R^d)),
\end{eqnarray*}
where we have used Theorem \ref{lemm4} in the last step. Hence,
Definition \ref{def-1.1} implies $$\mathrm{cap}(E,
\mathcal{BV}_{\mathcal{H}_{\alpha}}(\mathbb R^d))\lesssim
Cap^{\mathcal{H}_{\alpha}}_1(E).$$
\end{proof}

\subsection{Duality for $\alpha$-HBV capacity}\label{s4} In what follows,  we give the following lemma
 on the dual space $[\mathcal{BV}_{\mathcal{H}_{\alpha}}(\mathbb R^d)]^\ast$.
  Some similar results on various spaces have been obtained by some
  scholars in \cite{Ziemer}, \cite{xiao2} and \cite{liu}.

\begin{lemma}\label{l31} Let $\mu$ be a nonnegative Radon measure on $ \mathbb R^d$. Then the following two statements are
equivalent:

\item{\rm (i)}
$$
\Big|\int_{\mathbb R^d}f\,d\mu\Big|\lesssim
\big(\|f\|_{L^1}+\parallel \nabla_{\mathcal{H}_{\alpha}}
f\parallel({\mathbb R^d}) \big) \quad\forall\, f\in
\mathcal{BV}_{\mathcal{H}_{\alpha}}(\mathbb R^d).
$$

\item{\rm (ii)}
$$
\mu(B)\lesssim\,  {\mathrm{cap}}(B,
\mathcal{BV}_{\mathcal{H}_{\alpha}}(\mathbb R^d))
\quad\forall\,\hbox{Borel\ set}\, B\subseteq \mathbb R^d.
$$
\end{lemma}
\begin{proof}(i)$\Rightarrow$(ii).   For any compact set $K$,   taking $f=1_K$ in (i) and via the definition of
$P_{\mathcal{H}_{\alpha}}(K)$,  we have $$\mu(K)\lesssim \big(|K|+
P_{\mathcal{H}_{\alpha}}(K)\big).$$ For all bounded open sets $O$
with smooth boundary in $\mathbb R^d$ containing $K$, via the
definition of the $\alpha$-Hermite perimeter, we have
$$P_{\mathcal{H}_{\alpha}}(\bar{O})=P_{\mathcal{H}_{\alpha}}(O) $$
due to $|\bar{O}\setminus O|=0$.  Using the assumption,  we obtain
$$ \mu(O) \le  \mu(\bar{O})  \lesssim \big(|\bar{O}|+P_{\mathcal{H}_{\alpha}}(\bar{O})\big)= |O|+
P_{\mathcal{H}_{\alpha}}(O).$$

 Via Theorem \ref{lemm4},   we have
$$ \mu(K) \lesssim \mathrm{cap}(K, \mathcal{BV}_{\mathcal{H}_{\alpha}}(\mathbb R^d)).$$
  Corollary \ref{cor8} and the regularity of
$\mu$ yield
$$ \mu(B) \lesssim \mathrm{cap}(B, \mathcal{BV}_{\mathcal{H}_{\alpha}}(\mathbb R^d)) $$
holds for any Borel set $B\subseteq \mathbb R^d$.

(ii)$\Rightarrow$(i). Suppose (ii) is true. Firstly, we claim that
$f$ is finite almost everywhere  with respect to the measure $\mu$
for $f\in \mathcal{BV}_{\mathcal{H}_{\alpha}}(\mathbb R^d)$. Indeed,
we can assume $f\in C^\infty_c(\mathbb R^d) \bigcap
\mathcal{BV}_{\mathcal{H}_{\alpha}}(\mathbb R^d)$. For $t>0$, let
$E_t=\{x\in \mathbb R^d: |f(x)|>t\}$. By the co-area formula
(\ref{eq2.8}),   we know $E_t$ has finite perimeter for a.e. $t$ and
$$\int^\infty_0 P_{\mathcal{H}_{\alpha}}(E_t)dt\approx\| \nabla_{\mathcal{H}_{\alpha}} |f|\|(\mathbb R^d)<\infty.$$
From this,  we conclude that $\lim \inf_{t\rightarrow \infty}
P_{\mathcal{H}_{\alpha}}(E_t)=0.$
 Via Theorem \ref{lemm4}, we have
$$\mathrm{cap}(\{x\in \mathbb R^d:\ |f(x)|=\infty \}, \mathcal{BV}_{\mathcal{H}_{\alpha}}(\mathbb R^d))\lesssim \lim \inf_{t\rightarrow \infty} \{|E_t|+P_{\mathcal{H}_{\alpha}}(E_t)\}=0.$$ By the
assumption,  we know $\mu(\{x\in  \mathbb R^d:\ |f(x)|=\infty \})=0.$
This completes the proof of the claim.

If $f\in C^\infty_c(\mathbb R^d)
\bigcap\mathcal{BV}_{\mathcal{H}_{\alpha}}(\mathbb R^d)$, combining
the layer-cake formula, Theorem \ref{lemm4} and the co-area formula
(\ref{eq2.8}),  we obtain
\begin{eqnarray*}
\Big|\int_{ \mathbb R^d}f\,d\mu\Big|&\le& \int_{0}^\infty\mu\big(\{x\in  \mathbb R^d:\ |f(x)|>t\}\big)\,dt\\
&\lesssim&   \int_{0}^\infty\hbox{cap}\big(\{x\in  \mathbb R^d:\ |f(x)|>t\},\mathcal{BV}_{\mathcal{H}_{\alpha}}(\mathbb R^d)\big)\,dt\\
&\lesssim&   \int_{0}^\infty\Big\{|\{x\in  \mathbb R^d:\
|f(x)|>s\}|+
P_{\mathcal{H}_{\alpha}}\big(\{x\in  \mathbb R^d:\ |f(x)|>s+\varepsilon\}\big)\Big\}\,ds\\
&\lesssim&\|f\|_{L^1}+\parallel \nabla_{\mathcal{H}_{\alpha}} |f| \parallel({\mathbb R^d})\\
&\lesssim& \|f\|_{L^1}+ \parallel
\nabla_{\mathcal{H}_{\alpha}} f
\parallel({\mathbb R^d}),
\end{eqnarray*}
and so (i) follows for all $f\in
\mathcal{BV}_{\mathcal{H}_{\alpha}}(\mathbb R^d)$ via Theorem
\ref{propos4}.
\end{proof}

\begin{theorem}
\label{t31} If $E\subseteq \mathbb R^d$ is a Borel set, then
$$
 {\mathrm{cap}}(E,
\mathcal{BV}_{\mathcal{H}_{\alpha}}(\mathbb
R^d))=\sup_{\mu\in\mathscr{M}}\mu(E),
$$
where $\mathscr{M}$ is the class of all nonnegative Radon measures
$\mu\in\big[\mathcal{BV}_{\mathcal{H}_{\alpha}}(\mathbb
R^d)\big]^\ast$ with
$\|\mu\|_{\big[\mathcal{BV}_{\mathcal{H}_{\alpha}}(\mathbb
R^d)\big]^\ast}\le 1$.
\end{theorem}
\begin{proof}
In terms of Corollary \ref{cor8} (ii), we only need to verify that
  the formula holds for any compact set $K\subseteq \mathbb R^d$. Given a
compact set $K\subseteq \mathbb R^d$. Suppose:

(i) $\mathscr{X}$ is the set of all nonnegative Radon measures $\mu$
with support being contained in $K$ and $\mu( \mathbb R^d)=1$;

(ii) $\mathscr{Y}$ is the class of
$\mathcal{BV}_{\mathcal{H}_{\alpha}}(\mathbb R^d)$ functions $f$
with $$\|f\|_{L^1}+\| \nabla_{\mathcal{H}_{\alpha}} f\|({\mathbb
R^d})\le 1.$$

Then, $\mathscr{X}$ and $\mathscr{Y}$ are convex, $\mathscr{X}$ is
compact in the weak-star topology, and $\mu\mapsto\int_{ \mathbb
R^d}f\,d\mu$ is lower semicontinuous on $\mathscr{X}$ for each given
$f\in \mathscr{Y}$.

Note that if $f\in
\mathcal{A}(K,\mathcal{BV}_{\mathcal{H}_{\alpha}}(\mathbb R^d))$ and
$\mu\in \mathscr{M}$,  then
$$
\mu(K)\le\int_{ \mathbb R^d}f\,d\mu\le\Big(
\|f\|_{L^1}+\nabla_{\mathcal{H}_{\alpha}} f\|({\mathbb
R^d})\Big)\|\mu\|_{\big[\mathcal{BV}_{\mathcal{H}_{\alpha}}(\mathbb
R^d)\big]^\ast}\le \|f\|_{L^1}+\| \nabla_{\mathcal{H}_{\alpha}}
f\|({\mathbb R^d}).
$$
Hence,
$$
\mu(K)\le \hbox{cap}(K,\mathcal{BV}_{\mathcal{H}_{\alpha}}(\mathbb
R^d))
$$
 implies
$$
\hbox{cap}(K,\mathcal{BV}_{\mathcal{H}_{\alpha}}(\mathbb
R^d))\ge\sup_{\mu\in\mathscr{M}}\mu(K).
$$
To verify the reverse inequality, we observe two facts below. At
first, via
$$
\sup_{f\in \mathscr{Y}}\int_{ \mathbb
R^d}f\,d\mu=\|\mu\|_{\big[\mathcal{BV}_{\mathcal{H}_{\alpha}}(\mathbb
R^d)\big]^\ast},
$$
 we conclude that
$$
\inf_{\mu\in \mathscr{X}}\sup_{f\in \mathscr{Y}}\int_{ \mathbb
R^d}f\,d\mu=\inf_{\mu\in
\mathscr{X}}\|\mu\|_{\big[\mathcal{BV}_{\mathcal{H}_{\alpha}}(\mathbb
R^d)\big]^\ast}\ge\inf_{\mu\in\mathscr{M}} \big({\mu(K)}\big)^{-1}.
$$
Secondly,
$$
\inf_{\mu\in \mathscr  X}\int_{\mathbb R^d}f\,d\mu=\inf_{x\in
K}f(x)
$$
 implies
$$
\sup_{f\in \mathscr{Y}}\inf_{\mu\in \mathscr X}\int_{ \mathbb
R^d}f\,d\mu=\sup_{f\in \mathscr{Y}}\inf_{x\in
K}f(x)\le\Big(\hbox{cap}(K,\mathcal{BV}_{\mathcal{H}_{\alpha}}(\mathbb
R^d))\Big)^{-1}.
$$
Now, using  \cite[Theorem 2.4.1]{Admas},  we have
$$
\big({\sup_{\mu\in\mathscr{M}}\mu(K)}\big)^{-1}=\inf_{\mu\in\mathscr{M}}
\big({\mu(K)}\big)^{-1}\le\big({\hbox{cap}(K,\mathcal{BV}_{\mathcal{H}_{\alpha}}(\mathbb
R^d))}\big)^{-1},
$$
which completes the proof of the theorem.
\end{proof}

\subsection{Trace and $\alpha$-HBV isocapacity inequality}\label{s5}

Similar to  \cite[Theorem 10]{xiao2}, we  obtain the
trace/restriction theorem arising from
$\mathcal{BV}_{\mathcal{H}_{\alpha}}(\mathbb R^d)$.

\begin{theorem}\label{thm18} Given $1\le p<\infty$ and a nonnegative
Radon measure $\mu$ on $\mathbb R^d$. The following three statements
are equivalent:
\item{{\rm(i)}}  For any $f\in \mathcal{BV}_{\mathcal{H}_{\alpha}}(\mathbb R^d)$,
$$\Big(\int_{\mathbb R^d}|f|^p d\mu\Big)^{{1}/{p}}\lesssim
 \|f\|_{L^1}+\parallel \nabla_{\mathcal{H}_{\alpha}} f
\parallel(\mathbb R^d) .$$

\item{{\rm(ii)}} For any Borel set $B\subseteq \mathbb R^d$,
$$\mu(B)^{{1}/{p}}\lesssim \Big(|B|+ P_{\mathcal{H}_{\alpha}}(B)\Big).$$

\item{{\rm(iii)}} For any Borel set $B\subseteq \mathbb R^d$,
$$\mu(B)^{{1}/{p}}\lesssim   \mathrm{cap}(B, \mathcal{BV}_{\mathcal{H}_{\alpha}}(\mathbb
R^d)).$$
\end{theorem}
\begin{proof}
(i)$\Rightarrow$(ii). By taking $f=1_B$ and the definition of
$P_{\mathcal{H}_{\alpha}}(\cdot)$, we can deduce that (ii) is valid.

(ii)$\Rightarrow$(iii). For all bounded open sets
$O\subseteq \mathbb R^d$ with smooth boundary containing
$B$ which is a compact subset, using the assumption  we obtain
$$(\mu(O))^{1/p}\le (\mu(\bar{O}))^{1/p}\lesssim \Big( |\bar{O}|+P_{\mathcal{H}_{\alpha}}(\bar{O})\Big)=  \Big(|{O}|+ P_{\mathcal{H}_{\alpha}}(O)\Big).$$
 Theorem \ref{lemm4}  implies
$$(\mu(B))^{1/p}\lesssim  \mathrm{cap}(B, \mathcal{BV}_{\mathcal{H}_{\alpha}}(\mathbb R^d))\approx   \inf_{\mathfrak{g}\supseteq B}\Big\{P_{\mathcal{H}_{\alpha}}(O)+|O|\Big\}.$$
Then following from (ii) of Corollary \ref{cor8} and the inner
regularity of $\mu$,   we conclude that  (iii) is true.

(iii)$\Rightarrow$(i). Suppose (iii) holds. If $f\in
C^\infty_0(\mathbb R^d)\bigcap
\mathcal{BV}_{\mathcal{H}_{\alpha}}(\mathbb R^d)$, similar to the
proof of Lemma \ref{l31},  we know that $f$ is finite a.e.   with
respect to the measure $\mu$ for $f\in C^\infty_c(\mathbb
R^d)\bigcap \mathcal{BV}_{\mathcal{H}_{\alpha}}(\mathbb R^d)$.
Combining with the layer-cake formula, Theorem \ref{lemm4} and the
co-area formula (\ref{eq2.8}),  we conclude that
\begin{eqnarray*}
\Big(\int_{\mathbb R^d}
|f|^pd\mu\Big)^{1/p}&\le& \Big(\int^\infty_0 \mu(\{x\in
\mathbb R^d: |f(x)|>t\})dt^p\Big)^{1/p}\\
&\le& \int^\infty_0 \frac{d}{dt}\Big(\int^t_0 \mu(\{x\in  \mathbb
R^d:
|f(x)|>s\})ds^p\Big)^{1/p}dt\\
&=&  \int^\infty_0  \Big(\int^t_0 \mu(\{x\in  \mathbb R^d:
|f(x)|>s\})ds^p\Big)^{1/p-1}\mu(\{x\in  \mathbb R^d:
|f(x)|>t\})t^{\alpha-1}dt\\
&\le& \int^\infty_0  \Big(  \mu(\{x\in  \mathbb R^d: |f(x)|>t\})
\Big)^{1/p}dt\\
&\lesssim&   \int^\infty_0 \mathrm{cap}(\{x\in  \mathbb R^d:
|f(x)|>t\},
\mathcal{BV}_{\mathcal{H}_{\alpha}}(\mathbb R^d))dt\\
 &\lesssim&   \int_{0}^\infty \Big[\Big|\Big\{x\in  \mathbb R^d:\
 |f(x)|>s\Big\}\Big|+
P_{\mathcal{H}_{\alpha}}\big(\{x\in  \mathbb R^d:\ |f(x)|>s\}\big)\Big]\,ds\\
&\lesssim&   \|f\|_{L^1}+\parallel \nabla_{\mathcal{H}_{\alpha}} |f|\parallel(\mathbb R^d) \\
&\lesssim&    \|f\|_{L^1}+ \parallel
\nabla_{\mathcal{H}_{\alpha}} f
\parallel(\mathbb R^d) .
\end{eqnarray*}
Hence, combining  Theorem \ref{propos4} with the above proofs we
know that (i) is true.
\end{proof}

If  $\mu$ in the above theorem is taken as the Lebesgue measure, we
can obtain the the following  imbedding result for the
$\alpha$-Hermite case.

\begin{theorem}\label{thm1}
\item{{\rm (i)}} For any $f\in L^{{d}/{(d-1)}}(\mathbb R^d)$   with compact support, the analytic inequality
\begin{equation}\label{eq5}
\parallel f\parallel_{ {d}/{(d-1)}}\lesssim\Big(\int^{\infty}_0
\big(\mathrm{cap}(\{x\in  \mathbb R^d: |f(x)|\ge t\},
\mathcal{BV}_{\mathcal{H}_{\alpha}}(\mathbb
R^d))\big)^{{{d}/{(d-1)}}}dt^{{{d}/{(d-1)}}}\Big)^{{{(d-1)}/{d}}}
\end{equation}
 is equivalent to  the
geometric inequality
\begin{equation}\label{eqq6}
|M|^{{{(d-1)}/{d}}}\lesssim \mathrm{cap}(M, \mathcal{BV}_{\mathcal{H}_{\alpha}}(\mathbb R^d))
\end{equation}
for any compact set $M$ in $ \mathbb R^d$. Moreover, the
inequalities $\mathrm{(\ref{eq5})}$ and $\mathrm{(\ref{eqq6})}$
 are true.

\item{{\rm (ii)}}\label{thm2}  For any $f\in C^{1}_c(\mathbb R^d)$, the analytic inequality
\begin{eqnarray}\label{eq7}
&&\Big(\int^{\infty}_0\big(\mathrm{cap}(\{x\in  \mathbb R^d: |f(x)|\ge
t\}, \mathcal{BV}_{\mathcal{H}_{\alpha}}(\mathbb
R^d))\big)^{{{d}/{(d-1)}}}dt^{{{d}/{(d-1)}}}\Big)^{1-1/d}\\
&&\leq
\int_{\mathbb R^d}|  f(x)|dx+ \int_{\mathbb
R^d}|\nabla_{\mathcal{H}_{\alpha}} f(x)|dx\nonumber
\end{eqnarray}
 is equivalent to  the
geometric inequality
\begin{equation}\label{eq8}
\mathrm{cap}(M, \mathcal{BV}_{\mathcal{H}_{\alpha}}(\mathbb R^d))\le
|M|+ P_{\mathcal{H}_{\alpha}}(M)
\end{equation}
for any connected compact set $M$ in $ \mathbb R^d$ with smooth
boundary. Moreover, the inequalities $\mathrm{(\ref{eq7})}$ and
$\mathrm{(\ref{eq8})}$ are true.

\end{theorem}
\begin{proof} We adopt the method in \cite{xiao1} to give the proof.
In what follows,  we always adopt two short notations:
$$\Omega_t(f)=\{x\in  \mathbb R^d:\ |f(x)|\ge t\}$$
and $$\partial\Omega_t(f)=\{x\in  \mathbb R^d:\ |f(x)|=t\}$$ for a
function $f$ defined on $ \mathbb R^d$ and a number $t>0$.

(i). Given a compact set $M\subseteq \mathbb R^d,$ let $f=1_M$. Then
$\parallel f\parallel_{{d}/{(d-1)}}=|M|^{1-1/d}$
and $$ \Omega_t(f)=\left\{\begin{array}{cc} M, \hspace{0.4cm}
\hbox{if} \hspace{0.2cm}
t\in (0,1],\\
\ \ \emptyset,\hspace{0.4cm}\   \hbox{if}\hspace{0.2cm} t\in
(1,\infty).
\end {array}\right. $$
Hence,
\begin{eqnarray*}
&&\int^\infty_0 \Big[ \mathrm{cap}(\Omega_t(f), \mathcal{BV}_{\mathcal{H}_{\alpha}}(\mathbb R^d))\Big]^{{d}/{(d-1)}}dt^{{d}/{(d-1)}}\\
&&\quad =\int^1_0 ( \mathrm{cap}(\Omega_t(f),
\mathcal{BV}_{\mathcal{H}_{\alpha}}(\mathbb
R^d)))^{{d}/{(d-1)}}dt^{{d}/{(d-1)}}+\int^\infty_1(
\mathrm{cap} (\Omega_t(f)),
\mathcal{BV}_{\mathcal{H}_{\alpha}}(\mathbb R^d))^{{d}/{(d-1)}}dt^{{d}/{(d-1)}}\\
&&\quad=( \mathrm{cap}(M,
\mathcal{BV}_{\mathcal{H}_{\alpha}}(\mathbb
R^d)))^{{d}/{(d-1)}},
\end{eqnarray*}
which derives that  (\ref{eq5}) implies (\ref{eqq6}).

Conversely,  we show that (\ref{eqq6}) implies (\ref{eq5}). Suppose
(\ref{eqq6}) holds for any compact set in $ \mathbb R^d$. For $t>0$
and $f$, an $L^{{d}/{(d-1)}}$ integrable   function with compact
support in $ \mathbb R^d$, we use the
  inequality (\ref{eqq6}) to get
$$\Big\|
f\Big\|^{{d}/{(d-1)}}_{{d}/{(d-1)}}=\int^{\infty}_0|\Omega_t(f)|dt^{{d}/{(d-1)}}\lesssim
 \int^{\infty}_0\big( \mathrm{cap}(\Omega_t(f),
\mathcal{BV}_{\mathcal{H}_{\alpha}}(\mathbb
R^d))\big)^{{d}/{(d-1)}}dt^{{d}/{(d-1)}}.$$

Since (\ref{eq5}) is equivalent to (\ref{eqq6}), it suffices to
prove that (\ref{eqq6}) is valid. In fact, for any bounded set $B$
with smooth boundary  containing $M$,  using (ii) of Theorem
\ref{thm2.7}, we have
$$|M|^{1-1/d}\le |B|^{1-1/d}\lesssim P_{\mathcal{H}_{\alpha}}(B)\le\big(|B|+P_{\mathcal{H}_{\alpha}}(B)\big).$$
Theorem \ref{lemm4} implies that  (\ref{eqq6}) holds true.

(ii) For   a connected compact set $M\subseteq \mathbb R^d$ with
smooth boundary, let $R>0$ be such that $M\subseteq B(0,R)$.
 Choose $\delta>0$ such that
$2\delta<{\mathrm{dist}_{\mathbb R^d}({M},\partial B(0,R))}$, where
$\mathrm{dist}_{\mathbb R^d}(M,\partial
 B(0,R))$ represents the Euclidean  distance from ${M}$ to
$ B(0,R)$.

Define the Lipschitz function $$
f_\delta(x)=\left\{\begin{array}{cc}
1-\delta^{-1}{\mathrm{dist}_{\mathbb R^d}
(x,{M})},& \hspace{0.4cm} \hbox{if} \hspace{0.2cm}
 \mathrm{dist}_{\mathbb R^d}(x,{M})<\delta;\\
  0,&\hspace{0.4cm}\   \hbox{if}\hspace{0.2cm} \mathrm{dist}_{\mathbb R^d}(x,{M})\ge\delta.
 \end {array}
 \right. $$

Let $A_\delta$ be the intersection of $ B(0,R)$ with a tubular
neighborhood of $M$ of radius $\delta$. If (\ref{eq7}) holds, then
due to $M\subseteq\Omega_t(f_\delta)$ for $t\in [0,1]$,
$$ \mathrm{cap}(M, \mathcal{BV}_{\mathcal{H}_{\alpha}}(\mathbb R^d))\le \Big(\int^{1}_0
\big( \mathrm{cap}(\Omega_t(f_\delta),
\mathcal{BV}_{\mathcal{H}_{\alpha}}( \mathbb
R^d))\big)^{{d}/{(d-1)}}dt^{{d}/{(d-1)}}\Big)^{1-1/d}\le
\parallel
 f_\delta\parallel_{L^1}+
\parallel
 \nabla_{\mathcal{H}_{\alpha}}f_\delta\parallel_{L^1}.$$

 Via the coarea
 formula   and (i) in Lemma \ref{lem1.1}, we have
\begin{eqnarray*}\parallel \nabla_{\mathcal{H}_{\alpha}}f_\delta\parallel_{L^1}
&\approx&\int^1_0 P_{\mathcal{H}_{\alpha}}(\{x\in \mathbb{R}^d:
f_\delta(x)>t
\})dt\\
&=&\int^1_0 P_{\mathcal{H}_{\alpha}}(\{x\in \mathbb{R}^d:
\mathrm{dist}_{\mathbb
R^d}(x,{M})<\delta(1-t) \})dt\\
&=&\int^1_0 \big(P_{\mathcal{H}_{\alpha}}(\{x\in \mathbb{R}^d:
\mathrm{dist}_{\mathbb R^d}(x,{M})<\delta(1-t)
\})-P_{\mathcal{H}_{\alpha}}(M)\big)dt+P_{\mathcal{H}_{\alpha}}(M).
\end{eqnarray*} Next, we deal with the following integral
$$\int^1_0
\big(P_{\mathcal{H}_{\alpha}}(\{x\in \mathbb{R}^d:
\mathrm{dist}_{\mathbb R^d}(x,{M})<\delta(1-t)
\})-P_{\mathcal{H}_{\alpha}}(M)\big)dt.$$

By Lemma \ref{lem3.6}, we have
\begin{eqnarray*}&&\int^1_0 \big(P_{\mathcal{H}_{\alpha}}(\{x\in \mathbb{R}^d:
\mathrm{dist}_{\mathbb R^d}(x,{M})<\delta(1-t) \})-P_{\mathcal{H}_{\alpha}}(M)\big)dt\\
&&\le \int^1_0 \big(P_{\mathcal{H}_{\alpha}}(\{x\in \mathbb{R}^d:
0<\mathrm{dist}_{\mathbb R^d}(x,{M})<\delta(1-t) \})\big)dt.
\end{eqnarray*}
Denote by
$$E_\delta=\Big\{x\in \mathbb{R}^d: 0<\mathrm{dist}_{\mathbb
R^d}(x,{M})<\delta(1-t) \Big\}, $$ and by $\mathcal {F}'$  the class of all
functions $\varphi'=(\varphi'_1,\varphi'_2,\cdots,\varphi'_{d})\in
C^1_c(B(0,R);\mathbb R^{d})$  such that $$\parallel
\varphi\parallel_{\infty}=\sup_{x\in B(0,R)}(\mid
\varphi'_1(x)\mid^2+\cdots+\mid \varphi'_{d}(x)\mid^2)^{1/2}\le 1.$$
Then \begin{eqnarray*}
P_{\mathcal{H}_{\alpha}}(E_\delta)=\|\nabla_{\mathcal{H}_{\alpha}}
1_{E_\delta}\|(B(0,R))&=&\sup_{\varphi\in \mathcal {F}}\Big\{\int_{E_\delta}
\mathrm{div}_{\mathcal{H}_{\alpha}}\varphi(x)dx\Big\}\\
&\lesssim& \sup_{\varphi\in \mathcal {F}'}\Big\{\int_{E_\delta} \mathrm{div} \varphi'(x)dx\Big\}+ \sup_{\varphi\in \mathcal
{F}'}\Big\{\int_{E_\delta}
 x\cdot \varphi'(x)dx\Big\}\\
&\lesssim&
\Big(P(E_\delta)+\|x\|_{L^\infty(B(0,R))}|E_\delta|\Big)\rightarrow
0\end{eqnarray*} via letting $\delta\rightarrow 0$, where
$P(E_\delta)$ is the classical perimeter of $E_\delta$ and we also
have used  the fact on page 125 in \cite{Mazya}. Therefore, we know
that $P(E_\delta)\rightarrow 0$ when $\delta\rightarrow 0$. Hence,
$\parallel
\nabla_{\mathcal{H}_{\alpha}}f_\delta\parallel_{L^1}\rightarrow
P_{\mathcal{H}_{\alpha}}(M)$ when $\delta\rightarrow 0.$

We also have
\begin{eqnarray*}\parallel f_\delta\parallel_{L^1} &=&
 \frac{1}{\delta}\Big(\int^{\delta}_0 \Big|\Big\{x\in \mathbb{R}^d:
\mathrm{dist}_{\mathbb R^d}(x,{M})<s \Big\}\Big| ds\Big).
\end{eqnarray*}
Then  we  conclude that  $$\parallel
 f_\delta\parallel_{L^1}+
\parallel
 \nabla_{\mathcal{H}_{\alpha}}f_\delta\parallel_{L^1}\rightarrow |M|+P_{\mathcal{H}_{\alpha}}(M) $$ via letting $\delta\rightarrow 0$. Hence,  (\ref{eq8})
is valid.

Suppose (\ref{eq8}) is true for any connected compact set $M$ in $
\mathbb R^d$ with smooth boundary. By the  monotonicity of $
\mathrm{cap}(\cdot, \mathcal{BV}_{\mathcal{H}_{\alpha}}(\mathbb
R^d))$, we conclude that  $t\rightarrow
 \mathrm{cap}(\Omega_t(f), \mathcal{BV}_{\mathcal{H}_{\alpha}}(\mathbb R^d))$ is a decreasing function on
$[0,\infty)$. Then
\begin{eqnarray*}
t^{\frac{1}{d-1}}( \mathrm{cap}(\Omega_t(f), \mathcal{BV}_{\mathcal{H}_{\alpha}}(\mathbb R^d))^{\frac{d}{d-1}}&=&\Big[t \mathrm{cap}(\Omega_t(f), \mathcal{BV}_{\mathcal{H}_{\alpha}}(\mathbb R^d))\Big]^{\frac{1}{d-1}} \mathrm{cap}(\Omega_t(f), \mathcal{BV}_{\mathcal{H}_{\alpha}}(\mathbb R^d))\\
&\le& \Big(\int^t_0   \mathrm{cap}(\Omega_r(f), \mathcal{BV}_{\mathcal{H}_{\alpha}}(\mathbb R^d))dr\Big)^{\frac{1}{d-1}} \mathrm{cap}(\Omega_t(f), \mathcal{BV}_{\mathcal{H}_{\alpha}}(\mathbb R^d))\\
&=&(1-1/d)\frac{d}{dt}\Big(\int^t_0
  \mathrm{cap}(\Omega_r(f), \mathcal{BV}_{\mathcal{H}_{\alpha}}(\mathbb R^d))dr\Big)^{d/(d-1)}.
\end{eqnarray*}

Via (\ref{eq8}) and the above estimate,   we have
\begin{eqnarray*}
&&\int^\infty_0 ( \mathrm{cap}(\Omega_t(f),
\mathcal{BV}_{\mathcal{H}_{\alpha}}(\mathbb
R^d)))^{d/(d-1)}dt^{d/(d-1)}\\
&&\quad=
\frac{d}{d-1}\int^\infty_0
( \mathrm{cap}(\Omega_t(f), \mathcal{BV}_{\mathcal{H}_{\alpha}}(\mathbb R^d)))^{d/(d-1)}t^{{1}/{(d-1)}}dt\\
&&\quad\le \int^\infty_0 \Big[\frac{d}{dt}\Big(\int^t_0
  \mathrm{cap}(\Omega_r(f), \mathcal{BV}_{\mathcal{H}_{\alpha}}(\mathbb R^d))dr\Big)^{d/(d-1)}\Big]dt\\
&&\quad= \Big(\int^\infty_0  \mathrm{cap}(\Omega_t(f),
\mathcal{BV}_{\mathcal{H}_{\alpha}}(\mathbb R^d))
dt\Big)^{d/(d-1)}\\
&&\quad\le  \Big(\int^\infty_0
|\Omega_t(f)|+P_{\mathcal{H}_{\alpha}}(\Omega_t(f))
dt\Big)^{d/(d-1)}\\
&&\quad\approx \Big(\int_{\mathbb R^d}|f|+|
\nabla_{\mathcal{H}_{\alpha}}f|dx\Big)^{d/(d-1)},
\end{eqnarray*} where
we have used the co-area formula (\ref{eq2.8}) in the last step.

Similarly, since   (\ref{eq7}) is equivalent to (\ref{eq8}), it
suffices to check that (\ref{eq8}) is valid for any connected
compact set $M$ in $ \mathbb R^d$ with smooth boundary.   (ii) of
Theorem \ref{lemm4} implies that (\ref{eq8}) is valid.

\end{proof}

\section{$\alpha$-Hermite mean curvature}\label{sec-5}
In this section  we focus on the question whether  every set of
finite restricted $\alpha$-Hermite perimeter in $\mathbb R^{d}$ has
mean curvature in $L^{1}(\mathbb R^{d})$. For the classical case,
please refer to \cite{Barozzi} for the details.

For a given $u\in L^{1}(\mathbb R^{d})$, the Massari type functional
corresponding to the restricted $\alpha$-Hermite perimeter  is defined as
\begin{equation}\label{eqq5.1}\mathscr{F}_{u,
\mathcal{H}_{\alpha}}(E):=\widetilde{P}_{\mathcal{H}_{\alpha}}(E)+\int_{E}u(x)dx,\end{equation}
where $E$ is an arbitrary set of finite restricted $\alpha$-Hermite perimeter in $\mathbb
R^{d}$.

\begin{theorem}\label{thm5-1}
For every set $E$ of finite restricted $\alpha$-Hermite perimeter in $\mathbb
R^{d}$, there exists a function $u\in L^{1}(\mathbb R^{d})$ such
that
$$\mathscr{F}_{u, \mathcal{H}_{\alpha}}(E)\leq \mathscr{F}_{u, \mathcal{H}_{\alpha}}(F)$$
holds for every set $F$ of finite restricted $\alpha$-Hermite perimeter in
$\mathbb R^{d}$.
\end{theorem}
\begin{proof} Although the result under $\alpha=1$ goes back to the result of \cite{Barozzi} and may be treated as an application of \cite[Theorem 3.1]{BM}, it is still of some interest to present a demonstration.

At first,  for the given set $E$, we need to find a  function $u\in
L^{1}(\mathbb R^{d})$   such that
\begin{equation}\label{eq5.1}
\mathscr{F}_{u, \mathcal{H}_{\alpha}}(E)\leq \mathscr{F}_{u,
\mathcal{H}_{\alpha}}(F)
\end{equation}
holds for every $F$ with either $F\subset E$ or $E\subset F$, then
Theorem \ref{thm5-1} is proved, i.e. (\ref{eq5.1}) holds for every
$F\subset \mathbb R^{d}$. In fact,  by adding the inequalities
(\ref{eq5.1}) corresponding to the test sets $E\cap F$ and $E\cup
F$, we get
$$\begin{cases}
\widetilde{P}_{\mathcal{H}_{\alpha}}(E)+\int_{E}u(x)dx\leq \widetilde{P}_{\mathcal{H}_{\alpha}}(E\cap F)+\int_{E\cap F}u(x)dx;\\
\widetilde{P}_{\mathcal{H}_{\alpha}}(E)+\int_{E}u(x)dx\leq
\widetilde{P}_{\mathcal{H}_{\alpha}}(E\cup F)+\int_{E\cup F}u(x)dx.
\end{cases}$$
Then  noting that
\begin{equation*}\label{eq5.3}
\widetilde{P}_{\mathcal{H}_{\alpha}}(E\cap
F)+\widetilde{P}_{\mathcal{H}_{\alpha}}(E\cup F)\le
\widetilde{P}_{\mathcal{H}_{\alpha}}(E)+\widetilde{P}_{\mathcal{H}_{\alpha}}(F),
\end{equation*}
we can get
\begin{eqnarray*}
2\widetilde{P}_{\mathcal{H}_{\alpha}}(E)+2\int_{E}u(x)dx&\leq&\widetilde{P}_{\mathcal{H}_{\alpha}}(E\cap F)+\widetilde{P}_{\mathcal{H}_{\alpha}}(E\cup F)+\int_{E\cap F}u(x)dx+\int_{E\cup F}u(x)dx\\
&\leq&\widetilde{P}_{\mathcal{H}_{\alpha}}(E)+\widetilde{P}_{\mathcal{H}_{\alpha}}(F)+\int_{E}u(x)dx+\int_{F}u(x)dx,
\end{eqnarray*}
that is, (\ref{eq5.1}) holds for arbitrary $F$. Also, if
(\ref{eq5.1}) holds for $F\subset E$, then   for the sets $F$ such
that $E\subset F$, i.e.   $F^c\subset E^c$,
\begin{eqnarray*}
\widetilde{P}_{\mathcal{H}_{\alpha}}(E)+\int_{E}u(x)dx&=&\widetilde{P}_{\mathcal{H}_{\alpha}}(E^c)+\int_{E^c}u(x)dx-\int_{E^c}u(x)dx+\int_{E}u(x)dx\\
&\leq& \widetilde{P}_{\mathcal{H}_{\alpha}}(F^c)+\int_{F^c}u(x)dx-\int_{E^c}u(x)dx+\int_{E}u(x)dx\\
&=& \widetilde{P}_{\mathcal{H}_{\alpha}}(F)+\int_{F^c}u(x)dx-\int_{E^c}u(x)dx+\int_{E}u(x)dx\\
&=&\mathscr{F}_{u,
\mathcal{H}_{\alpha}}(F)-\int_{F}u(x)dx+\int_{F^c}u(x)dx-\int_{E^c}u(x)dx+\int_{E}u(x)dx\\
&=&\mathscr{F}_{u, \mathcal{H}_{\alpha}}(F)-\int_{F/
E}u(x)dx-\int_{E^C/ F^C}u(x)dx\\
&=&\mathscr{F}_{u, \mathcal{H}_{\alpha}}(F),\end{eqnarray*} where we
have used the fact that $u(\cdot)$ vanishes outside the set $E$.
Hence, we only need to prove that   $u$ defined on $E$ is integrable
and (\ref{eq5.1}) holds for any $F\subset E$.

{\it Step I.} Denote by  $h(\cdot)$   a measurable function
satisfying that $h>0$ on $E$ and $\int_{E}h(x)dx<\infty$, and denote
by $\Lambda$ the (positive and totally finite) measure:
\begin{equation*}
\Lambda(F)=\int_{F}h(x)dx,\ F\subset E.
\end{equation*}
It is obvious that  $\Lambda(F)=0$ if and only if $|F|=0$. For
$\lambda>0$ and $F\subset E$, consider the functional
$$\mathscr{F}_{\lambda}(F):=\widetilde{P}_{\mathcal{H}_{\alpha}}(F)+\lambda\Lambda(E\setminus F).$$
It is well known that  every minimizing sequence is compact in
$L^{1}_{loc}(\mathbb R^{d})$ and the functional is
lower-semicontinuous with respect to the same convergence. Hence, we
conclude that, for every $\lambda>0$, a solution $E_{\lambda}$ to
the problem:
$$\mathscr{F}_{\lambda}(F)\rightarrow \text{min},\ F\subset E.$$
Choose a sequence $\{\lambda_{i}\}$ of positive numbers, strictly
increasing to $\infty$, and denote the corresponding solutions by
$E_{i}\equiv E_{\lambda_{i}}$, so that $\forall i\geq 1$:
\begin{equation}\label{eq5.2}
\mathscr{F}_{\lambda_{i}}(E_{i})\leq \mathscr{F}_{\lambda_{i}}(F)\ \forall\ F\subset E.
\end{equation}
Given $i<j$. Let $F=E_{i}\cap E_{j}$. It follows from (\ref{eq5.2}) that $$\mathscr{F}_{\lambda_{i}}(E_{i})\leq \mathscr{F}_{\lambda_{i}}(E_{i}\cap E_{j}),
$$ that is,
$$\widetilde{P}_{\mathcal{H}_{\alpha}}(E_{i})+\lambda_{i}\Lambda(E\setminus E_{i})
\leq \widetilde{P}_{\mathcal{H}_{\alpha}}(E_{i}\cap
E_{j})+\lambda_{i}\Lambda(E\setminus(E_{i}\cap E_{j})),$$ which
implies
$$\widetilde{P}_{\mathcal{H}_{\alpha}}(E_{i})+\lambda_{i}\int_{E\setminus E_{i}}h(x)dx\leq \widetilde{P}_{\mathcal{H}_{\alpha}}(E_{i}\cap E_{j})+\lambda_{i}\int_{E\setminus(E_{i}\cap E_{j})}h(x)dx.$$
A direct computation gives
$$\widetilde{P}_{\mathcal{H}_{\alpha}}(E_{i})\leq \lambda_{i}\int_{E_{i}\setminus E_{j}}h(x)dx+ \widetilde{P}_{\mathcal{H}_{\alpha}}(E_{i}\cap E_{j}).$$
On the other hand, taking $F= E_{i}\cup E_{j}\subset E$ in (\ref{eq5.2}), we can get
 $\mathscr{F}_{\lambda_{j}}(E_{j})\leq \mathscr{F}_{\lambda_{j}}(E_{i}\cup E_{j})$.
 Hence,
$$\widetilde{P}_{\mathcal{H}_{\alpha}}(E_{j})+\lambda_{j}\int_{E\setminus E_{j}}h(x)dx\leq \widetilde{P}_{\mathcal{H}_{\alpha}}(E_{i}\cup E_{j})
+\lambda_{j}\int_{E\setminus(E_{i}\cup E_{j})}h(x)dx,$$
equivalently,
$$\widetilde{P}_{\mathcal{H}_{\alpha}}(E_{j})+\lambda_{j}\int_{E_{i}\setminus E_{j}}h(x)dx\leq \widetilde{P}_{\mathcal{H}_{\alpha}}(E_{i}\cup E_{j})$$
which implies that
$$\widetilde{P}_{\mathcal{H}_{\alpha}}(E_{i})+\widetilde{P}_{\mathcal{H}_{\alpha}}(E_{j})+\lambda_{j}\int_{E_{i}\setminus E_{j}}h(x)dx\leq \widetilde{P}_{\mathcal{H}_{\alpha}}(E_{i}\cup E_{j})+\lambda_{i}\int_{E_{i}\setminus E_{j}}h(x)dx+ \widetilde{P}_{\mathcal{H}_{\alpha}}(E_{i}\cap E_{j}).$$
Recall that $h>0$. The above estimate, together with (\ref{eq5.3}) and the facts that $\lambda_{i}<\lambda_{j}$, indicates that
$$(\lambda_{j}-\lambda_{i})\Lambda(E_{i}\setminus E_{j})=(\lambda_{j}-\lambda_{i})\int_{E_{i}\setminus E_{j}}h(x)dx=0,$$
that is, $E_{i}\subset E_{j}$ and the sequence of minimizers $\{E_{i}\}$ is increasing. On the other hand, letting $F=E$, we get
$$\widetilde{P}_{\mathcal{H}_{\alpha}}(E_{i})+\lambda_{i}\Lambda(E\setminus E_{i})\leq \widetilde{P}_{\mathcal{H}_{\alpha}}(E)+
\lambda_{i}\Lambda(E\setminus
E)=\widetilde{P}_{\mathcal{H}_{\alpha}}(E)\ \forall i\geq 1,$$ which
deduces that $E_{i}$ converges monotonically and in
$L^{1}_{loc}(\mathbb R^{d})$ to $E$. Via Lemma \ref{le-1.17}, we get
$$\begin{cases}
\widetilde{P}_{\mathcal{H}_{\alpha}}(E)\leq\lim\inf\limits_{i\rightarrow \infty}\widetilde{P}_{\mathcal{H}_{\alpha}}(E_{i})\leq \widetilde{P}_{\mathcal{H}_{\alpha}}(E),\\
\widetilde{P}_{\mathcal{H}_{\alpha}}(E)\leq\lim\inf\limits_{i\rightarrow
\infty}\widetilde{P}_{\mathcal{H}_{\alpha}}(E_{i})\leq\lim\sup\limits_{i\rightarrow
\infty}\widetilde{P}_{\mathcal{H}_{\alpha}}(E_{i})\leq
\widetilde{P}_{\mathcal{H}_{\alpha}}(E),
\end{cases}$$
which means
\begin{equation}\label{eq5.4}
\widetilde{P}_{\mathcal{H}_{\alpha}}(E)=\lim_{i\rightarrow
\infty}\widetilde{P}_{\mathcal{H}_{\alpha}}(E_{i}).
\end{equation}

{\it Step II.} Let $\lambda_{0}=0$ and $E_{0}=\emptyset$, and define
$$u(x)=
\begin{cases}
-\lambda_{i}\cdot h(x),&\ x\in E_{i}\backslash E_{i-1}, i\geq1;\\
0,&\ \text{otherwise}.
\end{cases}$$
Clearly, $u$ is negative almost everywhere on $E$, and
\begin{eqnarray*}
\int_{\mathbb R^{d}}|u(x)|dx&=&\int_{\cup^{\infty}_{i=0}E_{i+1}\backslash E_{i}}|u(x)|dx\\
&=&\sum^{\infty}_{i=0}\int_{E_{i+1}\backslash E_{i}}\lambda_{i+1}\cdot h(x)dx\\
&=&\sum^{\infty}_{i=0}\lambda_{i+1}\Lambda(E_{i+1}\backslash E_{i}).
\end{eqnarray*}
In (\ref{eq5.2}), taking $F=E_{i+1}$, we have
$$\widetilde{P}_{\mathcal{H}_{\alpha}}(E_{i})+\lambda_{i}\Lambda(E\setminus E_{i})\leq \widetilde{P}_{\mathcal{H}_{\alpha}}(E_{i+1})+\lambda_{i}\Lambda(E\setminus E_{i+1}),$$
that is, for every $i\geq 0$,
$$\lambda_{i}\Lambda(E_{i+1}\backslash E_{i})\leq \widetilde{P}_{\mathcal{H}_{\alpha}}(E_{i+1})-\widetilde{P}_{\mathcal{H}_{\alpha}}(E_{i}).$$
Then for sufficiently large $N$,
$$\sum^{N}_{i=0}\lambda_{i}\Lambda(E_{i+1}\backslash E_{i})\leq \sum^{N}_{i=0}\Big[\widetilde{P}_{\mathcal{H}_{\alpha}}(E_{i+1})-\widetilde{P}_{\mathcal{H}_{\alpha}}(E_{i})\Big]=\widetilde{P}_{\mathcal{H}_{\alpha}}(E_{N})
-\widetilde{P}_{\mathcal{H}_{\alpha}}(E_{0})=\widetilde{P}_{\mathcal{H}_{\alpha}}(E_{N}).$$
Letting $N\rightarrow \infty$, (\ref{eq5.4}) indicates that
\begin{equation*}
\sum^{\infty}_{i=0}\lambda_{i}\Lambda(E_{i+1}\backslash E_{i})\leq
\widetilde{P}_{\mathcal{H}_{\alpha}}(E).
\end{equation*}
We make the additional assumption that $0<\lambda_{i+1}-\lambda_{i}\leq c, i\geq 0$, where $c$ is a constant independent of $i$.
Then for any $N>0$,
\begin{eqnarray*}
\sum^{N}_{i=0}(\lambda_{i+1}-\lambda_{i})\Lambda(E_{i+1}\backslash E_{i})&\leq&c\sum^{N}_{i=0}\Lambda(E_{i+1}\backslash E_{i})\\
&=&c \sum^{N}_{i=0}\int_{E_{i+1}\backslash E_{i}}h(x)dx\\
&=&c\int_{\cup^{N}_{i=0}(E_{i+1}\backslash E_{i})}h(x)dx,
\end{eqnarray*}
which gives
$$\sum^{\infty}_{i=0}(\lambda_{i+1}-\lambda_{i})\Lambda(E_{i+1}\backslash E_{i})\leq c\Lambda(E).$$
Then
\begin{eqnarray*}
\int_{\mathbb R^{d}}|u(x)|dx&=&\sum^{\infty}_{i=0}\lambda_{i+1}\Lambda(E_{i+1}\backslash E_{i})\\
&=&\sum^{\infty}_{i=0}(\lambda_{i+1}-\lambda_{i})\Lambda(E_{i+1}\backslash E_{i})+\sum^{\infty}_{i=0}\lambda_{i}\Lambda(E_{i+1}\backslash E_{i})\\
&\leq&c\Lambda(E)+\widetilde{P}_{\mathcal{H}_{\alpha}}(E)<\infty.
\end{eqnarray*}
In conclusion, $u\in L^{1}(\mathbb R^{d})$.

{\it Step III.} We claim that for every $i\geq 1$ the inequality
\begin{equation}\label{eq5.5}
\widetilde{P}_{\mathcal{H}_{\alpha}}(E_{i})\leq
\widetilde{P}_{\mathcal{H}_{\alpha}}(F)+\sum^{i}_{j=1}\lambda_{j}\Lambda((E_{j}\backslash
E_{j-1})\backslash F)
\end{equation}
holds for any $F\subset E$.

For $i=1$, $E_{i-1}=E_{0}=\emptyset$. Then (\ref{eq5.5}) becomes
$$\widetilde{P}_{\mathcal{H}_{\alpha}}(E_{1})\leq \widetilde{P}_{\mathcal{H}_{\alpha}}(F)+\lambda_{1}\Lambda(E_{1}\backslash F),$$
which coincides with (\ref{eq5.2}) for $i=1$.

Now we assume that (\ref{eq5.5}) holds for a fixed $i\geq 1$ and
every $F\subset E$. Take $F\cap E_{i}$ as a test set. Note that
$\{E_{j}\}$ is increasing. It is easy to see that $$(E_{j}\backslash
E_{j-1})\backslash (F\cap E_{i})=(E_{j}\backslash E_{j-1})\backslash
F).
$$ Then
\begin{eqnarray*}
\widetilde{P}_{\mathcal{H}_{\alpha}}(E_{i})&\leq& \widetilde{P}_{\mathcal{H}_{\alpha}}(F\cap E_{i})+\sum^{i}_{j=1}\lambda_{j}\Lambda((E_{j}\backslash E_{j-1})\backslash (F\cap E_{i}))\\
&=&\widetilde{P}_{\mathcal{H}_{\alpha}}(F\cap
E_{i})+\sum^{i}_{j=1}\lambda_{j}\Lambda((E_{j}\backslash
E_{j-1})\backslash F).
\end{eqnarray*}
On the other hand, $E_{i+1}$ is a minimizer of $\mathscr{F}_{\lambda_{i+1}}$. Hence,
$$\mathscr{F}_{\lambda_{i+1}}(E_{i+1})\leq \mathscr{F}_{\lambda_{i+1}}(F\cup E_{i}),$$
and noticing that $$E\backslash E_{i}=(E\backslash E_{i+1})\cup (E_{i+1}\backslash E_{i}),$$ we can get
$$E\backslash(F\cup E_{i})= ((E\backslash E_{i+1})\backslash F)\cup ((E_{i+1}\backslash E_{i})\backslash F).$$
This gives
\begin{eqnarray*}
\widetilde{P}_{\mathcal{H}_{\alpha}}(E_{i+1})+\lambda_{i+1}\Lambda(E\backslash E_{i+1})&\leq&\widetilde{P}_{\mathcal{H}_{\alpha}}(F\cup E_{i})+\lambda_{i+1}\Lambda(E\backslash(F\cup E_{i}))\\
&\leq&\widetilde{P}_{\mathcal{H}_{\alpha}}(F\cup
E_{i})+\lambda_{i+1}\Lambda((E\backslash E_{i+1})\backslash
F)+\lambda_{i+1}\Lambda((E_{i+1}\backslash E_{i})\backslash F).
\end{eqnarray*}
Therefore, we obtain that
\begin{eqnarray*}
&&\widetilde{P}_{\mathcal{H}_{\alpha}}(E_{i})+\widetilde{P}_{\mathcal{H}_{\alpha}}(E_{i+1})+\lambda_{i+1}\Lambda(E\backslash E_{i+1})\\
&&\leq \widetilde{P}_{\mathcal{H}_{\alpha}}(F\cap E_{i})+\sum^{i}_{j=1}\lambda_{j}\Lambda((E_{j}\backslash E_{j-1})\backslash F)\\
&&+\widetilde{P}_{\mathcal{H}_{\alpha}}(F\cup E_{i})+\lambda_{i+1}\Lambda((E\backslash E_{i+1})\backslash F)+\lambda_{i+1}\Lambda((E_{i+1}\backslash E_{i})\backslash F)\\
&&\leq \widetilde{P}_{\mathcal{H}_{\alpha}}(E_{i})+\widetilde{P}_{\mathcal{H}_{\alpha}}(F)+\sum^{i+1}_{j=1}\lambda_{j}\Lambda((E_{j}\backslash E_{j-1})\backslash F)+\lambda_{i+1}\Lambda((E\backslash E_{i+1})\backslash F)\\
&&\leq
\widetilde{P}_{\mathcal{H}_{\alpha}}(E_{i})+\widetilde{P}_{\mathcal{H}_{\alpha}}(F)+\sum^{i+1}_{j=1}\lambda_{j}\Lambda((E_{j}\backslash
E_{j-1})\backslash F)+\lambda_{i+1}\Lambda(E\backslash E_{i+1}),
\end{eqnarray*}
that is, (\ref{eq5.5}) holds for $i+1$. Finally,
\begin{eqnarray*}
\widetilde{P}_{\mathcal{H}_{\alpha}}(E)&=&\lim_{i\rightarrow\infty}\widetilde{P}_{\mathcal{H}_{\alpha}}(E_{i})\\
&\leq& \widetilde{P}_{\mathcal{H}_{\alpha}}(F)+\lim_{i\rightarrow\infty}\sum^{i}_{j=1}\lambda_{j}\Lambda((E_{j}\backslash E_{j-1})\backslash F)\\
&=&\widetilde{P}_{\mathcal{H}_{\alpha}}(F)-\int_{\cup^{\infty}_{i=0}(E_{j}\backslash E_{j-1})\backslash F}u(x)dx\\
&=&\widetilde{P}_{\mathcal{H}_{\alpha}}(F)-\int_{E\backslash F}u(x)dx,
\end{eqnarray*}
 which gives (\ref{eq5.2}).

\end{proof}

\begin{remark}\label{remark-5.1}
In Definition \ref{equa5.1}, taking $\alpha=1$, it is obvious that
(\ref{equa5.1}) holds for all $$\varphi\in C^1_0(\mathbb
R^{d};\mathbb R^{2d}) \ \text{with}\ \ \|\varphi\|_{\infty}\leq 1,
$$
namely,
$
\mathcal{F}_R(\mathbb R^{d})=\mathcal{F}(\mathbb R^{d}).
$

\end{remark}

\vspace{0.5cm}

\noindent{\bf Acknowledgements}. The authors are grateful for Prof. Jie Xiao for many helpful and useful suggestions. In addition:

\begin{itemize}
\item { Jizhen Huang was supported  by
Fundamental Research Funds for the Central Universities (\#
500419772)}.

\item {Pengtao Li was in part supported by National Natural Science Foundation of
China (\# 11871293\ \&\ \# 11571217) \& Shandong Natural Science
Foundation of China (\# ZR2017JL008, \# ZR2016AM05)}.

\item {Yu Liu was supported by National
Natural Science Foundation of China (\# 11671031) \& Beijing
Municipal Science and Technology Project (\# Z17111000220000). }

\end{itemize}

\end{document}